\documentclass[letterpaper, 11pt,  reqno]{amsart}

\usepackage{amsmath,amssymb,amscd,amsthm,amsxtra, esint}
\usepackage{tikz} 
\usepackage{color}
\usepackage[unicode=true]{hyperref}
\usepackage{mathabx}
\usepackage{orcidlink}

\usepackage{cases}

\usepackage[left=32mm, right=32mm, 
bottom=27mm]{geometry}

\makeatletter
\renewcommand{\subjclassname}{%
  \textup{1991} Mathematics Subject Classification}
\@xp\let\csname subjclassname@1991\endcsname \subjclassname
\@namedef{subjclassname@2000}{%
  \textup{2000} Mathematics Subject Classification}
\@namedef{subjclassname@2010}{%
  \textup{2010} Mathematics Subject Classification}
  \@namedef{subjclassname@2020}{%
  \textup{2020} Mathematics Subject Classification}
\makeatother

\allowdisplaybreaks[2]

\usepackage{color}

\definecolor{gr}{rgb}   {0.,   0.69,   0.23 }
\definecolor{bl}{rgb}   {0.,   0.5,   1. }
\definecolor{mg}{rgb}   {0.85,  0.,    0.85}
\definecolor{yl}{rgb}   {0.8,  0.7,   0.}
\definecolor{or}{rgb}  {0.7,0.2,0.2}

\newcommand{\Cc}{\mathcal{C}}
\newcommand{\Bb}{\mathcal{B}}

\usepackage{tikz}
\usepackage{xcolor}
\usetikzlibrary{calc,intersections,through,backgrounds}
\usetikzlibrary{decorations.pathreplacing,decorations.pathmorphing,decorations.markings,decorations.shapes}
\usetikzlibrary{shapes}
\usetikzlibrary{external}

\newtheorem{theorem}{Theorem} [section]

\newtheorem{lemma}[theorem]{Lemma}
\newtheorem{proposition}[theorem]{Proposition}
\newtheorem{remark}[theorem]{Remark}

\newtheorem{definition}[theorem]{Definition}
\newtheorem{corollary}[theorem]{Corollary}

\DeclareMathOperator*{\supp}{supp}

\newcommand{\1}{\hspace{0.5mm}\text{I}\hspace{0.5mm}}
\newcommand{\II}{\text{I \hspace{-2.8mm} I} }
\newcommand{\III}{\text{I \hspace{-2.9mm} I \hspace{-2.9mm} I}}

\newcommand{\IV}{\text{I \hspace{-2.9mm} V}}

\newcommand{\noi}{\noindent}
\newcommand{\Z}{\mathbb{Z}}
\newcommand{\R}{\mathbb{R}}
\newcommand{\C}{\mathbb{C}}

\newcommand{\T}{\mathbb{T}}
\newcommand{\bul}{\bullet}
\newcommand{\deff}{\stackrel{\textup{def}}{=}}

\renewcommand{\T}{T^{\star}}

\let\Re=\undefined\DeclareMathOperator*{\Re}{Re}

\let\P= \undefined
\newcommand{\P}{\mathbf{P}}

\newcommand{\Q}{\mathbf{Q}}

\newcommand{\E}{\mathbb{E}}

\newcommand{\K}{\mathcal{K}}

\newcommand{\al}{\alpha}
\newcommand{\be}{\beta}
\newcommand{\dl}{\delta}

\newcommand{\vb}{\overline{v}}
\newcommand{\nb}{\nabla}

\newcommand{\too}{\longrightarrow}

\newcommand{\Dl}{\Delta_A}
\newcommand{\del}{\delta}
\newcommand{\eps}{\varepsilon}
\newcommand{\kk}{\kappa}
\newcommand{\g}{\gamma}
\newcommand{\G}{\Gamma}
\newcommand{\ld}{\lambda}
\newcommand{\Ld}{\Lambda}
\newcommand{\s}{\sigma}
\newcommand{\Si}{\Sigma}
\newcommand{\ft}{\widehat}

\newcommand{\cj}{\overline}

\newcommand{\dt}{\partial_t}
\newcommand{\dd}{\partial}

\newcommand{\Dlg}{\Delta_\gm}
\newcommand{\Dla}{\Delta_A}

\newcommand{\ta}{\theta}
\newcommand{\diag}{\bigtriangleup}

\newcommand{\ph}{\varphi}
\renewcommand{\l}{\ell}
\newcommand{\om}{\omega}
\renewcommand{\O}{\Omega}

\newcommand{\les}{\lesssim}

\newcommand{\jb}[1]
{\langle #1 \rangle}

\renewcommand{\S}{\mathcal{S}}

\newcommand{\gm}{\mathbf{\mathrm{g}}}

\newcommand{\Prob}{\mathbb{P}}

\newcommand{\M}{\mathcal{M}}
\def\e{\eps}

\newcommand{\N}{\mathbb{N}}

\renewcommand{\H}{\mathcal{H}}
\newcommand{\D}{\mathcal{D}}

\newcommand{\Id}{\textup{Id}}

\makeatletter
\def\DeclareSymbol#1#2#3{\expandafter\gdef\csname MH@symb@#1\endcsname{\tikz[baseline=#2, scale=.18]{#3}}}
\def\<#1>{\ensuremath{\mathchoice{\tikzsetnextfilename{macros#1}{\color{black}\csname MH@symb@#1\endcsname}}{\tikzsetnextfilename{macros#1}{\color{black}\csname MH@symb@#1\endcsname}}{\tikzsetnextfilename{macros#1}\scalebox{.7}{\color{black}\csname MH@symb@#1\endcsname}}
{\tikzsetnextfilename{macros#1}\scalebox{.5}{\color{black}\csname MH@symb@#1\endcsname}}}} 
\makeatother 
\DeclareSymbol{1}{0}{\path (-.4,0) -- (.4,0); \draw[line width=0.2mm] (0,0) -- (0,1.2) 
node[circle, fill, draw, solid, inner sep=0pt, minimum size=1.8pt] {};}

\newcommand{\PN}{\Psi_N}

\newcommand{\dg}{\mathbf{d}}

\newcommand{\VV}{\mathcal{V}}

\newtheorem*{ackno}{Acknowledgements}

\numberwithin{equation}{section}
\numberwithin{theorem}{section}

\begin{document}
\baselineskip = 15pt

\title[Stochastic SCGL on surfaces]
{Stochastic complex Ginzburg-Landau equation on compact surfaces}

\author[T.~Robert and Y.~Zine]
{Tristan Robert\orcidlink{0000-0003-1987-4536} and Younes Zine\orcidlink{0009-0001-7752-1205}}

\address{\small{Institut  \'Elie Cartan, Universit\'e de Lorraine, B.P. 70239, F-54506 Vand\oe uvre-l\`es-Nancy Cedex, France
}}

\email{tristan.robert@univ-lorraine.fr}

\address{
Younes Zine\\
 \'Ecole Polytechnique F\'ed\'erale de Lausanne\\
1015 Lausanne\\ Switzerland}

\email{younes.zine@epfl.ch}

\subjclass[2020]{35K15,60H15,58J35}

\keywords{stochastic equation on manifolds; 
polynomial nonlinearity;
Ginzburg-Landau equation;
Gibbs measure}

\begin{abstract} We study a stochastic complex Ginzburg-Landau equation (SCGL) on compact surfaces with magnetic Laplacian and polynomial nonlinearity, forced by a space-time white noise. After renormalizing the equation in a suitable manner, we show that the dynamics is locally well-posed. Moreover, we prove deterministic global well-posedness for the defocusing SCGL in the weakly dispersive regime. 
\end{abstract}



\maketitle

%


\tableofcontents

\baselineskip = 14pt

\section{Introduction}

\subsection{The stochastic complex Ginzburg-Landau equation}\label{SUBS:intro}
In this work, we discuss the well-posedness issue for the following stochastic complex Ginzburg-Landau (SCGL) equation with an additive space-time white noise forcing on a general closed (compact, connected, boundaryless) orientable surface: 
\begin{equation}
\begin{cases}
\dt u = (\alpha_1 + i\alpha_2) \Dla   u - (\beta_1 + i \beta_2) \VV'(|u|^2)u + \sqrt{2 \g} \xi, \quad (t,x) \in \R_+ \times \M  \\
u_{|t=0}=u_0,
\end{cases}
\label{SCGL 1}
\end{equation}
where $ u = u(t,x) $ is a complex-valued function, $\alpha_1>0$, $\g > 0$, $\alpha_2, \beta_1, \beta_2 \in \R$, $(\M,\gm)$ is a closed orientable Riemannian surface, $\VV$ is a polynomial, and $\Dla$ is the magnetic Laplacian with magnetic potential given by a real-valued smooth 1-form $A\in C^\infty(\M;T^\star\M)$; see Subsection~\ref{SUBSEC:geom} below for proper definitions, and with stochastic forcing given by the complex-valued space-time white noise $\xi$.

On the one hand, the deterministic CGL equation (namely, \eqref{SCGL 1} with $\xi\equiv0$) arises in many physical situations, in particular superconductivity and Bose-Einstein condensates; see \cite{AK} for a review. On the other hand, the stochastic SCGL equation \eqref{SCGL 1} is also a test-bed model to incorporate random effects, and as such has already been investigated in \cite{B-S,CCL,Kuksin,KS,HairerICMP,Hoshino,HIN,Matsuda,Shirikyan,Trenberth,Younes} on different bounded spatial domains and with various stochastic source terms $\xi$; see also \cite{DDFa,DDFb,DDFc} for works on $\R^d$ with a confining potential.

While the aforementioned results eventually incorporate electric potentials, magnetic effects are not taken into account. Letting then $E$ be a complex line bundle over $(\M,\gm)$ with associated unitary connection $D_A$, the dynamics \eqref{SCGL 1} corresponds to the Hamiltonian/stochastic quantization equation associated with the Ginzburg-Landau energy
\begin{align}\label{YMHaction}
\H(u,A)=\int_\M \Big(\frac12|F_A|^2+\frac{\alpha}2|D_Au|^2+ \frac{\beta}2\VV(|u|^{2})\Big)dV,
\end{align}
defined for $u\in C^\infty(\M;E)$, a smooth section (identified with a function $u:\M\to\C$), where $F_A$ is the curvature of the unitary connection $D_A = d-iA$ (locally), $d$ is the exterior derivative, and $dV$ is the volume form on $(\M,\gm)$.

In this work, we are thus interested in giving a local and global solution theory for \eqref{SCGL 1}. However, as for the $P(\phi)_2$ EQFT \cite{Simon} and its stochastic quantization equation \cite{DPD}, the nonlinearity in \eqref{SCGL 1} needs a renormalization procedure to be made sense of; see Subsection~\ref{SUBSEC:renorm} below. Thus the model \eqref{SCGL 1} falls into the scope of singular stochastic PDEs. While the literature on this topic has been exponentially growing during the past decade, rather few contributions deal with non-flat spatial domains. We refer to \cite{Bailleul,BB1,BB2,BBF,BDFTa,BDM,BER,DDD,EMR,HairerSingh,Mouzard,MZ,ORT,ORTW} for other instances of singular parabolic, elliptic or hyperbolic singular PDEs on manifolds.

\subsection{Renormalization of the nonlinearity}\label{SUBSEC:renorm} 
Let us describe the renormalization procedure needed to make sense of the dynamics of \eqref{SCGL 1}. Indeed, since $\xi$ in \eqref{SCGL 1} is spatially rough, being a distribution in Sobolev or H\"older spaces of negative regularity, we expect $u$ to also be only a distribution in space. Thus $ \VV'(|u|^2)u$ is meaningless as it is. To be more precise, let us first look at the linearised equation
\begin{align}\label{linear}
\dt \Psi = (\alpha_1+i\alpha_2)\Dla \Psi +\sqrt{2\gamma}\xi.
\end{align}
Here, $\xi$ is a centred complex-valued Gaussian process on $\D'(\R; \D'(\M;\C)) $ (defined on some probability space endowed with a probability measure $\Prob$) with covariance function
\begin{align}\label{covarWN}
\E \Big[ \xi(\phi_1) \cj{ \xi(\phi_2) }   \Big] = \langle \phi_1 , \phi_2 \rangle_{L^2(\R\times\M)}.
\end{align}
for any $\phi_1,\phi_2 \in C_0^\infty(\R;C^\infty(\M))$, and $\langle \cdot, \cdot \rangle$ denotes the canonical inner product on $L^2(\R; L^2(\M;\C) )$. Thus we can extend the definition of $\xi$ to test functions belonging to $L^2(\R; L^2(\M;\C) )$ and define for $n \in \N$ the process 
\begin{align}\label{Bn}
B_n(t) = \langle \xi, \mathbf{1}_{[0,t]} \varphi_n \rangle_{L^2(\R\times\M)}, \quad t \geq 0,
\end{align}
where $\{\varphi_n\}_{n\ge 0}$ is an orthonormal basis  of $L^2(\M; \C)$ consisting of eigenfunctions of $-\Dla$ with corresponding eigenvalues $\{-\ld_n^2\}_{n \ge 0}$ which are assumed to be arranged in increasing order: $0\le \ld_0\le \ld_1\le ...\le \ld_n\to \infty$. Then from \eqref{covarWN} and \eqref{Bn} one sees that $\{B_n\}_{n \geq 0}$ is a family of mutually independent complex\footnote{This means that 
\begin{align}\label{BnRI}
B_n=\frac1{\sqrt{2}}(B_n^R+iB_n^I),
\end{align} 
with $B_n^R,B_n^I$ being standard real-valued Brownian motions independent of each other.} Brownian motions. It follows by construction that $\xi$ is the time derivative of the cylindrical Wiener process $W$ defined by
\begin{equation}\label{wiener}
W(t) = \sum_{n \geq 0} B_n(t) \varphi_n.
\end{equation}
Define 
\begin{align}\label{lambda+}
\lambda_+=\min\{\ld_n>0,~n\ge 0\},
\end{align}
and $\P_{\ge \lambda_+}$ the projection onto the space of positive eigenvalues\footnote{Indeed, $\ld_0=0$ when $A$ is exact.} frequencies:
\[ \P_{\ge \lambda_+} f = \sum_{\substack{n \ge 0 \\ \ld_n \ge \ld_+}} \jb{f, \phi_n}_{L^2} \, \varphi_n,\]
where $\jb{\cdot, \cdot}_{L^2}$ denotes the $L^2(\M, \C)$-inner product; see \eqref{inner} below. Then the\footnote{In the case $\ld_0\neq 0$, this choice makes $\Psi$ stationary as a stochastic process with values in $\Cc^s(\M)$, $s<0$; see Definition \ref{def} below.} solution to \eqref{linear} is given by the stochastic convolution
\begin{equation}\label{stochastic convolution}
\Psi(t) \deff \sqrt{2 \g} \int_{- \infty}^t S_{A,\alpha}(t-t') \P_{\ge \lambda_+}dW(t')+\mathbf{1}_{\ld_0=0}\sqrt{2\gamma}B_0(t)\varphi_0,
\end{equation}
where the semigroup
\begin{equation}
S_{A,\alpha}(t) = e^{t(\alpha_1 + i\alpha_2) \Dla }
\label{def semi-group}
\end{equation}
for $t \geq 0$ is defined through the functional calculus for $- \Dla$; see Subsection \ref{SUBSEC : analysis tools}. Then Proposition~\ref{PROP:construction stochastic objects} below shows that $\Psi\in C(\R;\Cc^{s}(\M))$ almost surely only if $s<0$. Since we expect that the solution $u$ to \eqref{SCGL 1} behaves as $\Psi$, and in particular only belongs to $C(\R;\Cc^s(\M))$ for $s<0$, there is an issue in making sense of the nonlinearity in \eqref{SCGL 1}.

To fix this issue, we thus need to renormalize the nonlinearity. From now on we focus on the case of a monomial nonlinearity: $$\VV(x)=\frac1{2m}x^m,$$ for an integer $m\ge 2$.
Then we regularize \eqref{SCGL 1} by applying a smooth frequency truncation on the noise term, and consider for any $N\in\N$ the equation
\begin{equation}\label{truncated SCGL}
\begin{cases}
\dt u_N = (\alpha_1 + i\alpha_2) \Dla  u_N - (\beta_1 + i \beta_2) | u _N|^{2m-2}u_N + \sqrt{2 \g} \P_N \xi\\
(u_N)_{|t=0}=u_0,
\end{cases}
\end{equation}
where $\P_N$ is a smooth projection onto frequencies smaller than $N$; see \eqref{PN} below. For each fixed $N\in\N$, \eqref{truncated SCGL} has a unique smooth global solution $u_N$ for $t>0$. From the previous discussion, $|u_N|^{2m-2}u_N$ diverges as $N\to \infty$, but we expect $u_N$ to behave like the linearization
\begin{equation}\label{truncated stochastic convolution}
\PN \deff \P_N \Psi.
\end{equation}
Now, for each $(t,x) \in \R_+ \times \M$, $\PN(t,x)$ is a mean-zero complex Gaussian variable, and using Itô isometry we can compute its variance:
\begin{align}\label{sigmaN}
\s_N(t,x) &= \E[| \PN (t,x) |^2 ] = \sum_{\ld_n\ge \lambda_+} \psi_0(\ld_n^2 N^{-2})^2 \frac{\g|\varphi_n(x)|^2}{\alpha_1 \ld_n^2}+\mathbf{1}_{\ld_0=0}2\gamma t|\varphi_0(x)|^2 \\
& = \frac{\gamma}{2\pi\alpha_1}\log N+O_{N,x}(1+t),\notag
\end{align}
for a cut-off function $\psi_0$ as in the definition of $\P_N$ in \eqref{PN}; the last asymptotics is a consequence of Lemma~\ref{LEM:green function 1} below. This shows how $\PN$ diverges as $N\to\infty$, for fixed $t$ and $x$. To fix this divergence in (the powers of) $\PN$, we define its Wick-ordered powers by
\begin{align}\label{renormalization stochastic objects}
\begin{split}
& :\PN^{j_1} \cj{\PN}^{j_2}:(t,x) \\
& \qquad \quad \deff 
\begin{cases} (-1)^{j_2} (j_2!) L_{j_2}^{(j_1-j_2)}( | \PN(t,x) |^2; \s_N(t,x) ) (\PN(t,x))^{j_1-j_2},& \quad j_1 \geq j_2 \\
(-1)^{j_1} (j_1!)  L_{j_1}^{(j_2-j_1)}( | \PN(t,x) |^2; \s_N(t,x) ) (\cj{\PN}(t,x))^{j_2-j_1},& \quad j_2 >j_1
\end{cases}
\end{split}
\end{align}
for any integers $j_1,j_2\ge 0$. Here, for $\ell,k \geq 0$, we write $L^{(\ell)}_k (x; \s^2) = \s^{2k} L_k^{(\ell)}(\frac{x}{\s^2}) $ where $L_k^{(\ell)}$ are the generalized Laguerre polynomials which can be defined through the generating function:
\begin{align*}
\frac{1}{(1-t)^{\ell+1}}e^{-\frac{tx}{1-t}} = \sum_{k \geq 0 } t^k L_k^{(\ell)}(x), \quad \text{for all $(t,x) \in (0,1) \times \R$}.
\end{align*}
We also note the following formula:
\begin{align}
\frac{d}{dx} L^{(\l)}_k(x) = - L^{(\l+1)}_{k-1}(x) \quad \text{for all $x \in \R$}
\label{YYY10}
\end{align}
for any $\l \ge 0$ and $k \ge 1$. We refer to \cite{OT1,Trenberth} for more details on the generalized Laguerre polynomials. In Section \ref{SEC : proba}, we study the convergence of the stochastic objects defined in \eqref{renormalization stochastic objects} above.

Finally, in order to get a well-defined nonlinearity at the level of the full truncated equation \eqref{truncated SCGL}, we replace the monomial nonlinearity $|u_N|^{2m-2}u_N$ by the Wick-ordered monomial:
\begin{equation}\label{renormalization nonlinearity}
: | u _N|^{2m-2}u_N :(t,x) \deff (-1)^{m-1} (m-1)! \, L_{m-1}^{(1)}(|u_N|^2(t,x); \s_N(t,x))u_N(t,x).
\end{equation}
Our proof strategy is then to the renormalized version of \eqref{truncated SCGL} study (see \eqref{renormalized truncated SCGL 1} below) via a first order expansion as in the celebrated work \cite{DPD}.

\subsection{Well-posedness of the renormalized dynamics}
In view of the discussion in the previous subsection, we thus consider the following equation:
\begin{equation}
\begin{split}
\dt u_N& = (\alpha_1 + i\alpha_2) \Dla u_N \\
&\qquad \quad - (\beta_1 + i \beta_2)(-1)^{m-1} (m-1)! \, L_{m-1}^{(1)}\big( |u_N|^2; \s_N \big)u_N   + \sqrt{2 \g} \P_N \xi,
\end{split}
\label{renormalized truncated SCGL 1}
\end{equation}
with some deterministic initial data $u_0$, for each $N \in \N$.

\smallskip

Our first result deals with local well-posedness for the renormalized equation \eqref{renormalized truncated SCGL 1}.
\begin{theorem} Let $\alpha_1,\gamma>0$ and $\alpha_2,\beta_1,\beta_2\in\R$. Fix $m \geq 2 $ an integer. Let $- \frac{2}{2m-1} < s_0 < 0$ and $0 < \e \ll 1$. Then, for any \eqref{renormalized truncated SCGL 1} is pathwise uniformly locally well-posed in $\Cc^{s_0}(\M;\C)$, uniformly in $N \in \N$. More precisely, for all $N \in \N$ and any given $u_0 \in \Cc^{s_0}(\M;\C)$, there exists a $\Prob$-almost surely positive stopping time $T$ independent of $N$, such that there exists a solution $u_N$ to \eqref{renormalized truncated SCGL 1}, unique in the class
\begin{align*}
\PN + C \left( [0,T]; \Cc^{s_0}(\M;\C) \right) \cap C \left((0,T]; \Cc^{2\e}(\M;\C) \right).
\end{align*}
Moreover, $\{u_N\}_{N \in \N}$ converges $\Prob$-almost surely to a non-trivial stochastic process ${u \in \Psi + C \left( [0,T]; \Cc^{s_0}(\M;\C) \right) \cap C \left((0,T]; \Cc^{2\e} (\M;\C) \right)}$ as $N$ goes to infinity.
\label{THM : LWP theory 1}
\end{theorem}

\begin{remark}\rm We point out the following facts.

\smallskip

\noi
\textup{(i)} Note that the previous result does not depend on the sign of $\beta_1$; namely, it applies to both focusing and defocusing equations. The global-in-time statement however, will only hold in the defocusing case ($\beta_1 > 0$), see Theorem \ref{THM : deterministic GWP} below.

\smallskip

\noi
\textup{(ii)} The limiting process $u$ provided by Theorem~\ref{THM : LWP theory 1} can then be interpreted as the unique solution $u$ in the class $ \Psi + C \left( [0,T]; \Cc^{s_0}(\M;\C) \right) \cap C \left((0,T]; \Cc^{2\e} (\M;\C) \right)$ to the renormalized equation
\begin{align}\label{WSCGL}
\dt u = (\alpha_1 + i\alpha_2) \Dla u - (\beta_1 + i \beta_2)\,:\!|u|^{2m}u\!:\,  + \sqrt{2 \g} \xi,
\end{align}
where the renormalized nonlinearity $\,:\!|u|^{2m}u\!:\,$ makes sense in the class $ \Psi + C \left( [0,T]; \Cc^{s_0}(\M;\C) \right) \cap C \left((0,T]; \Cc^{2\e} (\M;\C) \right)$. See Remark~1.3 in \cite{ORT}.
\end{remark}

We show next that we can extend the local flow constructed in Theorem \ref{THM : LWP theory 1} globally in time in some regime of parameters.

\begin{theorem}Let $\alpha_1, \beta_1, \gamma>0$ and $\alpha_2,\beta_2\in\R$. Fix $m \geq 2$ an integer and define the ratio $r = \frac{\alpha_1}{|\alpha_2|} \in(0,\infty]$.\footnote{Namely, we set $r = \infty$ if $\al_2 = 0$.} Let $0 < -s_0 <  \frac{2}{2m-1}$ and assume that
	\begin{equation}
2m-1 < 2 + 2 \left(   r^2 + r \sqrt{1 + r^2}  \right).
\label{condition on r}
\end{equation}
Then, the limit $u$ of the dynamics \eqref{renormalized truncated SCGL 1} provided by Theorem~\ref{THM : LWP theory 1} can be extended globally in time. More precisely, for any $u_0 \in \Cc^{s_0}(\M)$ and any target time $T >0$, there $\Prob$-almost surely exists a unique solution $u_N$ to \eqref{renormalized truncated SCGL 1} on $[0,T]$ for any $N \geq 1$. 
Moreover, $\{u_N\}_{N \in \N}$ converges $\Prob$-almost surely to a stochastic process $u \in \Psi + C \left( [0,T]; \Cc^{s_0}(\M) \right) \cap C \left((0,T]; \Cc^{2\e} (\M) \right)$ as $N$ goes to infinity.
\label{THM : deterministic GWP}
\end{theorem}

We conclude this introduction with several remarks.

\begin{remark}\label{RMK:expo}\rm The technical condition \eqref{condition on r} guarantees the existence of $p \ge 1$ such that
\begin{equation}\label{condition-p}
2m-1 < p < 2 + 2 \left(   r^2 + r \sqrt{1 + r^2}  \right).
\end{equation}
On the one hand, the lower bound $p > 2m-1$ in \eqref{condition on r} is such that we have an $L^p(\M)$ local well-posedness theory; see Proposition \ref{PROP:LWP2}. On the other hand, the bounds $2 - 2 \left(   r^2 + r \sqrt{1 + r^2}  \right) <p < 2 + 2 \left(   r^2 + r \sqrt{1 + r^2}  \right)$ allows us to get a suitable a priori estimate for our solution theory; see Lemmas \ref{LEM:a priori estimate} and \ref{LEM:a priori 2} below.
\end{remark}

\begin{remark}
\rm
We point out that our renormalization ``constant" $\sigma_N$ \eqref{sigmaN} appearing in \eqref{renormalization nonlinearity} actually depends on both $t\geq 0$ and $x\in\M$, as in \cite{ORT,OT1}. As emphasized in \cite{BDFTa}, it is more physically relevant to only have a ``locally covariant" Wick renormalization, i.e. independent of $t$ and $x$. Since we show in Lemma~\ref{LEM:green function 1} below the estimate $\sigma_N(t,x)=\frac{\gamma}{2\pi\alpha_1}\log N +O_{N,x}(1+t)$, in the cubic case $m=2$, we could indeed consider the alternative renormalized nonlinearity $(|u_N|^2-2C_N)u_N$ with $C_N=\frac{\gamma}{2\pi\alpha_1}\log N$, instead of ${\big(|u_N|^2-\sigma_N(t,x)\big)u_N}$. The difference between these two renormalization procedures, namely the term $(\sigma_N(t,x)-C_N)u_N$, is a linear perturbative term and does not affect our analysis. Hence, the local and global results in Theorems~\ref{THM : LWP theory 1} and~\ref{THM : deterministic GWP} hold for this alternative stationary renormalization procedure in the cubic case $m=2$. In the case of a higher order nonlinearity $m\geq 3$, it is less clear to us how to easily handle a locally covariant renormalization constant, so we sticked to the renormalization of the nonlinearity defined in \eqref{renormalization nonlinearity}. See also the discussion on this point in \cite{HairerSingh}.
\end{remark}

The geometric framework and analytic toolbox is recalled in Section~\ref{SEC : analysis tools} below. Section~\ref{SEC : proba} discusses the construction of the stochastic objects $:\Psi^{j_1}\overline{\Psi}^{j_2}:$ \eqref{renormalization stochastic objects}. The local (Theorem~\ref{THM : LWP theory 1}) and global (Theorem~\ref{THM : deterministic GWP}) well-posedness results are then discussed in Section~\ref{SEC : well-posedness }. 

\begin{ackno}\rm 
The authors would like to thank Tadahiro Oh for suggesting this problem.
T.R. was partially supported by DFG through the CRC 1283 ``Taming uncertainty and profiting from randomness
and low regularity in analysis, stochastics and their applications.'' and by the ANR project Smooth ANR-22-CE40-0017.
Y.Z. was partially supported by the European Research Council (grant no. 864138 “SingStochDispDyn”). Y.Z. and the chaire of probability and PDEs at EPFL.

\end{ackno}
\section{Analytic toolbox} \label{SEC : analysis tools}
\subsection{Notations and first properties}\label{SUBSEC:geom}
We write $ A \les B $ to denote an estimate of the form $ A \leq CB $. 
Similarly, we write  $ A \sim B $ to denote $ A \les B $ and $ B \les A $ and use $ A \ll B $ 
when we have $A \leq c B$ for small $c > 0$. We may write $A \les_\ta B$ for $A \leq C B$ with $C = C(\ta)$ 
if we want to emphasize the dependence of the implicit constant on some parameter $\ta$. 
We use $C, c > 0$, etc.~to denote various constants whose values may change line by line.

\subsection{Geometric background}

We now start by describing the geometric context, for which we refer to \cite{Jost}. Let $(\M,\gm)$ be a two-dimensional closed (compact, boundaryless) orientable smooth Riemannian manifold. We write $T\M$ and $T^*\M$ for its tangent and cotangent bundles. Throughout the paper, we fix a smooth metric $\gm$, which is a twice covariant tensor field, i.e. a smooth section of the bundle $T^*\M\otimes T^*\M$, given in local coordinates as (using Einstein summation convention)
$$\gm  =\gm_{j,k}dx_j\otimes dx_k,$$ with $(\gm_{j,k})_{j,k}$ smooth and taking value in the set of positive symmetric definite matrices. In particular $(\gm_{j,k})$ is invertible and its inverse is denoted by 
\begin{align*}
\big(\gm^{j,k}(x) \big) = \big(\gm_{j,k}(x)\big)^{-1}.
\end{align*}
We also write
\begin{align*}
|\gm(x)| = \det \big(\gm_{j,k}(x)\big)\ge c >0.
\end{align*}
Throughout this text, we write the volume form $dV = |\gm(x)|^{\frac12}dx$.

The inner product $\langle\cdot,\cdot\rangle_{\T\M}$ on vector fields is locally given by
$$\langle X,Y\rangle_{T\M}=\gm_{j,k}X^j\cj{Y^k}$$
for any vector fields $X=X^j\partial_j$, $Y=Y^k\partial_k$, i.e. smooth sections of $T\M$. Similarly, for any 1-forms $A=A_jdx^j$ and $B=B_kdx^k$, i.e. smooth sections of $T^*\M$, we have
$$\langle A,B\rangle_{T^*\M}=\gm^{j,k}A_j\cj{B_k}.$$

\subsection{Spectral properties of the magnetic Laplacian}
To each smooth unitary connection $D_A = d-iA: u\in C^\infty(\M)\mapsto (\partial_ju  -iuA_j)dx^j\in T^*\M$,  where $A_j$ are smooth, we then associate the magnetic Laplace-Beltrami operator locally defined by
\begin{align}\label{Delta}
\begin{split}
 \Dla  u & = D_A^*D_Au= |\gm|^{-\frac12}  \sum_{1 \leq j,k \leq 2}  \dd_j(|\gm|^{\frac12}\gm^{j,k}\dd_k u\big) \\
& \qquad \quad +i\gm^{j,k}A_k\partial_j u-i\gm^{j,k}\partial_j(A_ku)+(A_1^2+A_2^2)u
\end{split}
\end{align}
for any smooth function $u\in C^{\infty}(\M;\C)$. This defines a non-positive self-adjoint operator on $L^2(\M)$, which is known (see e.g. \cite{Shi}) to have a compact resolvent. Thus for each smooth $A$ there exists a basis $\{\varphi_n\}_{n\ge 0}\subset C^{\infty}(\M)$ of $L^2(\M;\C)$ consisting of complex-valued eigenfunctions of $-\Dla $ associated with the corresponding eigenvalues $\{-\ld_n^2\}_{n \ge 0}$ arranged in increasing order: $0\le \ld_0\le \ld_1\le ...\le \ld_n$ and satisfying $\ld_n\to\infty$ as $n \to \infty$. Moreover, any $u\in\D'(\M)$ can be expanded (in the sense of distributions) as
\begin{align*}
u = \sum_{n\ge 0}\langle u,\varphi_n\rangle_{\mathcal{D}',\mathcal{D}} \cdot \varphi_n,
\end{align*}
where $\langle\cdot,\cdot\rangle$ denotes the usual duality pairing between $\D'(\M)$ and $\D(\M)$ which coincides with the $L^2(\M;\C)$-inner product 
\begin{align}\langle u,\varphi_n\rangle_{L^2}=\int_\M u\cj{\varphi_n}dV
\label{inner}
\end{align}
when $u\in L^2(\M;\C)$. Last but not least, recall that the magnetic Laplacian $\Dla$ admits \emph{gauge invariance}: for any $f\in C^\infty(\M;\R)$ it holds $e^{-if}\Dla e^{if} = \Delta_{A+df}$ in the sense of operators \cite[Proposition 3.2]{Shi}. 
In the following, we write $A\sim B$ if $A$ and $B$ are gauge equivalent, namely $A=B+df$ for some smooth $f$. Shigekawa \cite[Theorem 4.2]{Shi} gives a geometric characterisation of gauge classes; of particular interest to us is that the lowest eigenvalue $\lambda_0$ of $\Dla$ is positive except when $A$ is exact.

In our analysis, we will need some information of the localization of the eigenvalues. This can be achieved through studying the spectral function, which is defined as 
\begin{align*}
e(x,y,\Ld^2) \deff \sum_{\ld_n^2 \leq \Ld^2} \ph_n(x)\cj{\ph_n(y)},
\end{align*}
for $(x,y)\in\M\times\M$ and $\Ld\in\R$. We have the following behaviour of the spectral function on the diagonal (see \cite[Theorem 1.1]{Hormander1}; note that here the $\ld_n$'s are the square roots of the eigenvalues of $-\Dla$).
\begin{lemma} \label{LEM:loc eig} 
The following asymptotics hold uniformly in $x\in\M$ as $\Ld\to\infty$:
\begin{equation} \label{loc eig}
e(x,x,\Ld^2) = c \Ld^2 + O(\Ld),
\end{equation}
for some constant $c>0$ only depending on $(\M,\gm)$.
\end{lemma}
In particular, integrating \eqref{loc eig} on $\M$, we obtain some useful information on the localization of the eigenvalues: 
\[ n+1 =   \# \big\{k: \ld_k \leq \ld_n \big\}  = c \ld_n ^ 2 + O(\ld_n) \]
This implies that 
\begin{align}\label{asympt ldn}
\ld_n \sim n^{\frac12}
\end{align}
 as $n\to\infty$. 

As a direct consequence, we obtain the boundedness in average of the eigenfunctions $\ph_n$.
\begin{corollary}\label{COR:avg eig}
For any $\nu\in\R$, there exists $C>0$ such that for any $\Ld>0$ and any $x\in\M$, it holds
\begin{align*}
\sum_{\ld_n\in (\Ld,\Ld+1]}\frac{|\varphi_n(x)|^2}{\ld_n^\nu} \le C\sum_{\ld_n\in (\Ld,\Ld+1]}\frac1{\ld_n^\nu}.
\end{align*}
\end{corollary}
\begin{proof}
We simply use Lemma \ref{LEM:loc eig} to bound
\begin{align*}
\sum_{\ld_n\in (\Ld,\Ld+1]}\frac{|\varphi_n(x)|^2}{\ld_n^\nu} &\sim \Ld^{-\nu}\Big[e\big(x,x,(\Ld+1)^2\big)-e(x,x,\Ld^2)\Big]\\
&\les \Ld^{-\nu}\Ld \les \Ld^{-\nu}\#\big\{k: \Ld<\ld_k \le \Ld+1\big\} \sim \sum_{\ld_n\in (\Ld,\Ld+1]}\ld_n^{-\nu}.
\end{align*}
\end{proof}

\subsection{Function spaces, linear and nonlinear estimates}\label{SUBSEC : analysis tools}
Next, recall that the functional calculus of $-\Dla$ is defined for any $\psi\in \S(\R;\C)$ by
\begin{align*}
\psi\big(-\Dla\big)f = \sum_{n\ge 0}\psi(\ld_n^2)\langle f,\varphi_n\rangle \varphi_n,
\end{align*}
for all $f\in C^{\infty}(\M;\C)$. This in particular allows us to define the more general class of Besov spaces associated with $\Dla$. First, using the functional calculus, we can define the Littlewood-Paley projectors $\Q_M$ for a dyadic integer $M\in 2^{\Z_{-1}}\deff \{\frac12,1,2,4,...\}$ as 
\begin{align*}
\Q_M = \begin{cases}
\psi_0\big(-M^{-2}\Dla\big)-\psi_0\big(-4M^{-2}\Dla\big),~~M\ge 1,\\
\psi_0\big(-4\Dla\big),~~M=\frac12,
\end{cases}
\end{align*}
where $\psi_0\in C^{\infty}_0(\R;\R)$ is non-negative and such that $\supp\psi_0\subset [-1,1]$ and $\psi_0\equiv 1$ on $[-\frac12,\frac12]$. 
\noi
We also define the regularization operators
\begin{align}\label{PN}
\P_N =\psi_0\big(-N^{-2}\Dla\big).
\end{align} 

With the inhomogeneous dyadic partition of unity $\{\Q_M\}_{M\in 2^{\Z_{-1}}}$, we can then define the Besov norms for $p,q \in [1,\infty]$ and $s\in \R$,
\begin{align}\label{Besov norm}
\| f \|_{\Bb^s_{p,q}(\M)} \deff \Big\|  \jb{M}^{qs} \| \Q_M  f \|_{L^p(\M)} \Big\|_{\ell^q_M} = \Big(\sum_{M\in 2^{\Z_{-1}}}\jb{M}^{qs}\big\|\Q_M f \big\|_{L^p(\M)}^q\Big)^{\frac1q}.
\end{align}
for any smooth function $f$. The last sum is of course understood to be the supremum for $q=\infty$.
\begin{definition}\label{def}\rm Fix $p,q \in [1,\infty]$ and $s \in \R$. We define the Besov space $\Bb^s_{p,q}(\M)$ to be the completion of $C^{\infty}(\M)$ for the norm $ \| \cdot \|_{\Bb^s_{p,q}(\M)} $. In particular, this defines the H\"older spaces $\Cc^s(\M) = \Bb^s_{\infty, \infty}(\M)$ for any $s\in \R$. 
\label{Besov}
\end{definition}

\smallskip

\begin{remark}\rm We make some observations.

\smallskip

\noi
(i) Note that since they only depend on the principal symbol of $\Dla$ (see e.g. \cite[Proposition 2.1]{BGT}), these spaces are independent of $A$.

\smallskip

\noi
(ii) In the literature, Besov spaces are often defined as the space of distributions having finite Besov norm (see e.g. \cite[Definition 2.68]{BCD} in the Euclidean case). Both definitions coincide for $p,q < \infty$, but there is a logarithmic difference when $p \text{ or } q  = \infty$, which is the case considered here due to the properties of the stochastic objects and the boundedness of the CGL propagator on these spaces (see Proposition~\ref{PROP:construction stochastic objects} and Lemma~\ref{LEM:semigroup} below). As it is more convenient to work with smooth functions in some places (see for example Lemma \ref{LEM:semigroup} \textup{(iii)}), we thus adopt Definition \ref{def} as our definition of Besov spaces. See also Remark 7 in \cite{MW2} for further discussion.
\end{remark}

In order to study various operators defined through the function calculus, such as the operators $\Q_M$, $\P_N$ defined above or the semigroup $S_{A,\alpha}(t)$ defined below, we first state a general bound on the kernel of such operators. We refer to \cite[Lemma 2.5]{ORTW} for the proof\footnote{In \cite{ORTW}, only real valued symbols are considered, but the proof extends to complex valued ones in a straightforward manner. Again, the proof is given for $-\Dlg$ but is the same for $-\Dla$ as it only depends on the principal symbol of the operator, which is the same in both cases.}.
\begin{lemma}\label{LEM:PM}
Let $\psi\in\S(\R;\C)$, and for any $h\in (0,1]$ define the kernel
\begin{align*}
\K_h(x,y)\deff \sum_{n\ge 0}\psi\big(h^{2}\ld_n^2\big)\varphi_n(x)\cj{\varphi_n(y)}.
\end{align*}  
Then for any $L>0$, there exists $C>0$ such that for any $h\in (0,1]$ and $x,y\in\M$ it holds
\begin{align}\label{K2}
\big|\K_h(x,y)\big|\le C h^{-2}\jb{h^{-1}\dg(x,y)}^{-L},
\end{align}
where $\dg$ is the geodesic distance on $\M$.
\end{lemma}
\noi
As in \cite{ORTW}, this implies the following boundedness property of such multipliers.
\begin{corollary}\label{COR : bernstein}
Let $\psi \in \S(\R;\C)$. For any $1\leq p\leq q\leq \infty$, there exists $C>0$ such that for any $u\in C^{\infty}(\M)$ and $h \in (0,1]$,
\begin{align*}
\big\|\psi(-h^2\Dla)u\big\|_{L^q(\M)}\le C h^{-2\big(\frac{1}{p}-\frac{1}{q}\big)}\|u\|_{L^p(\M)}.
\end{align*}
\end{corollary}
\noi

Next, recall that the semigroup $S_{A,\alpha}(t)$ for the linear complex (magnetic, time-dependent) Ginzburg-Landau equation has been defined in \eqref{def semi-group} above, and that $\alpha_1>0$. In particular, we have the following properties for $S_{A,\alpha}(t)$.
\begin{lemma}\label{LEM:semigroup}
 The following properties hold:\\
\textup{(i)} For any $1\le p_1\le p_2\le \infty$, any $1\le q\le\infty$ and $s_1,s_2\in\R$ with $s_1\le s_2$, there exists $C>0$ such that for any $0<t\le 1$ and any $f\in B^{s_2}_{p_1,q}(\M)$, we have the Schauder type estimate
\begin{align*}
\big\|S_{A,\alpha}(t) f\big\|_{\Bb^{s_2}_{p_2,q}}\le C t^{-\frac{(s_2-s_1)}{2} - (\frac{1}{p_1} - \frac{1}{p_2}) }\big\|f\big\|_{\Bb^{s_1}_{p_1,q}}.
\end{align*}
\textup{(ii)} Let $s_1,s_2 \in \R $ such that $s_1 \leq s_2 \le s_1 + 1 $. There exists $C>0$ such that for any $0<t\le 1$ and any $f\in \Cc^{s_1}(\M)$ , we have
\begin{align*}
\big\| \big( \operatorname{Id} - S_{A,\alpha}(t) \big)f \big\|_{C^{s_1}}\le C t^{ \frac{s_2 - s_1}{2} }\big\|f\big\|_{\Cc^{s_2}}.
\end{align*}
\textup{(iii)} Let $s \in \R$ and fix $f \in \Cc^{s}(\M)$. The map $t \in [0, \infty) \mapsto S_{A,\alpha}(t)f \in \Cc^{s}(\M)$ is continuous.
\end{lemma}
\begin{proof} The property (i) follows from the same argument as in the proof of Lemma 2.6 (ii) in \cite{ORTW}, since $\Dla$ and $\Dlg$ share the same principal symbol. 

To prove (ii), note that it suffices to show that for any smooth $f$,
\begin{align}\label{QS}
\big\|\Q_Mf-S_{A,\alpha}(t)\Q_Mf\big\|_{L^{\infty}(\M)}\les \big(\sqrt{t}\jb{M}\big)^{s_2-s_1}\big\|f\big\|_{L^{\infty}}
\end{align}
uniformly in $M\in 2^{\Z_{-1}}$. To prove this bound, we then follow the argument of the proof of \cite[Lemma 2.6 (ii)]{ORTW}: we write $S_{A,\alpha}(t)=\psi\big(-h^2\Dla\big)$ for some $\psi\in\S(\R;\C)$ satisfying $\psi(x)=e^{-(\alpha_1+i\alpha_2)x}$ for $x\ge 0$, and semi-classical parameter $h=\sqrt{t}$. In particular in view of Corollary~\ref{COR : bernstein} (applied to both $\Q_M$ and $S_{A,\alpha}(t)$) we only need to consider the case $\sqrt{t}\ll \jb{M}^{-1}$. Then we decompose locally $\Q_M$ and $S_{A,\alpha}(t)\Q_M$ as (see \cite{BGT,ORT,ORTW} for more details)
\begin{align*}
\kk^\star\big(\chi_0\Q_M\big) &=(\kk^\star\chi_0)\kk^\star\big(\chi_1\Q_M\big)=(\kk^\star\chi_0)\Big[\sum_{\ell=0}^{L-1}M^{-\ell}q_\ell(x,M^{-1}D)\kk^\star\chi_2+R_{-L,M}\Big]
\end{align*}
and
\begin{align*}
\kk^\star\big(\chi_0 S_{A,\alpha}(t)\Q_M\big)&=\sum_{k=0}^{K-1}t^{\frac{k}2}p_k(x,\sqrt{t}D)\Big[\sum_{\ell=0}^{L-1}M^{-\ell}q_\ell(x,M^{-1}D)\kk^\star\chi_2+R_{-L,M}\Big]+R_{-K,t}\Q_M,
\end{align*}
where $\chi_j$ are some smooth functions compactly supported within a chart $(U,V,\kk)$ on $\M$ with $\chi_{j+1}\equiv 1$ on $\supp\chi_j$, $q_\ell,p_k\in \S^{-\infty}(U\times\R^2)$ are some smooth symbols compactly supported in $x$, and also in $\eta$ for $q_\ell$, and the remainders satisfy
\begin{align}\label{remainder}
\big\|R_{-K,t}\big\|_{B^{-\s_1}_{2,2}(\M)\to B^{\s_2}_{2,2}(\R^2)} \les t^{\frac{K-\s_1-\s_2}2}
\end{align}
and
\begin{align*}
\big\|R_{-L,M}\big\|_{B^{-\s_1}_{2,2}(\M)\to B^{\s_2}_{2,2}(\R^2)} \les M^{\s_1+\s_2-L}
\end{align*}
for any $\s_1,\s_2\ge 0$ with $\s_1+\s_2\le \min\{K,L\}$. In particular, taking $\s_1,\s_2$ large enough so that $B^{\s_2}_{2,2}(\R^2)\subset L^{\infty}(\R^2)$ and $L^1(\M)\subset B^{-\s_1}_{2,2}(\M)$ (which is possible in view of Lemma~\ref{LEM:Besov}~(iv)) we can bound
\begin{align*}
\big\|R_{-K,t}\Q_M\big\|_{L^{\infty}(\M)\to L^{\infty}(\R^2)} &\les \big\|R_{-K,t}\Q_M\big\|_{B^{-\s_1}_{2,2}(\M)\to B^{\s_2}_{2,2}(\R^2)}\les t^{\frac{K-\s_1-\s_2}2},
\end{align*}
which is enough for \eqref{QS} provided that we take $K$ large enough. Similarly, we have
\begin{align*}
\big\|t^{\frac{k}2}p_k(x,\sqrt{t}D)R_{-L,M}\big\|_{L^{\infty}(\M)\to L^{\infty}(\R^2)} &\les t^{\frac{k}2}\big\|p_k(x,\sqrt{t}D)R_{-L,M}\big\|_{B^{-\s_1}_{2,2}(\M)(\M)\to B^{\s_2}_{2,2}(\R^2)}.
\end{align*}
Now, for $f\in B^{\s_2+10}_{2,2}(\R^2)$, using that $p_k\in \S^{-\infty}(U\times\R^2)$, integrations by parts and Cauchy-Schwarz inequality, we have
\begin{align*}
\big\|p_k(x,\sqrt{t}D)f\big\|_{B^{\s_2}_{2,2}(\R^2)}&\sim \Big\|\jb{\zeta}^{\s_2}\int_{\R^2}\int_{\R^2}e^{ix\cdot(\eta-\zeta)}p_k(x,\sqrt{t}\eta)\ft f(\eta)d\eta dx\Big\|_{L^2_\zeta}\\
&\sim \Big\|\jb{\zeta}^{\s_2}\int_{\R^2}\int_{\R^2}e^{ix\cdot(\eta-\zeta)}\jb{\eta-\zeta}^{-10}\jb{D_x}^{10}p_k(x,\sqrt{t}\eta)\ft f(\eta)d\eta dx\Big\|_{L^2_\zeta}\\
&\les \big\|\jb{\zeta}^{\s_2}\jb{\eta-\zeta}^{-10}\jb{\eta}^{-\s_2-10}\big\|_{L^2_{\zeta,\eta}}\big\|\jb{D_x}^{10}p_k(x,\sqrt{t}\eta)\big\|_{L^{\infty}_{x,\eta}}\|f\|_{B^{\s_2+10}_{2,2}(\R^2)}\\
&\les \|f\|_{B^{\s_2+10}_{2,2}(\R^2)}
\end{align*}
uniformly in $t$. Thus
\begin{align*}
\big\|t^{\frac{k}2}p_k(x,\sqrt{t}D)R_{-L,M}\big\|_{L^{\infty}(\M)\to L^{\infty}(\R^2)} &\les t^{\frac{k}2}\big\|R_{-L,M}\big\|_{B^{-\s_1}_{2,2}(\M)(\M)\to B^{\s_2+10}_{2,2}(\R^2)}\\
&\les t^{\frac{k}2}M^{\s_1+\s_2+10-L},
\end{align*}
which is again enough for \eqref{QS} if $L$ is large enough and $k\ge 1$, since we restricted ourselves to $s_2\le s_1 +1$. As for $k=0$, note that the principal symbol of $\kk^\star(\chi_0 S_{A,\alpha}(t))=\kk^\star\big(\chi_0 \psi(-t\Dla)\big)$ is 
\begin{align*}
(\kk^\star\chi_0)(x)\psi(t|\eta|_x^2)=(\kk^\star\chi_0)(x)e^{-t(\alpha_1+i\alpha_2)|\eta|^2_x}
\end{align*}
with $\alpha_1>0$, where we use the notation 
\begin{align*}
|\eta|^2_x =\gm^{j,k}(x)\eta_j\eta_k \sim |\eta|^2
\end{align*}
uniformly in $x\in\M$.
So we have for any smooth $f\in C^{\infty}(\M)$
\begin{align*}
&\Big\|\big[(\kk^\star \chi_0)\Id - p_0(x,\sqrt{t}D)\big]R_{-L,M}f\Big\|_{L^{\infty}(\R^2)}\\&=\Big\|(\kk^\star\chi_0)(x)\int_{\R^2}e^{ix\cdot\eta}\big[1-e^{-t(\alpha_1+i\alpha_2)|\eta|^2_x}\big]\widehat{(R_{-L,M}f)}(\eta)d\eta\Big\|_{L^{\infty}(\R^2)}\\
&\les t\int_{\R^2}\jb{\eta}^2\big|\widehat{(R_{-L,M}f)}(\eta)\big|d\eta \les t \big\|R_{-L,M}f\big\|_{B^4_{2,2}(\R^2)}\\
& \les t M^{s_1+4-L}\big\|f\big\|_{B^{-s_1}_{2,2}(\M)} \les t M^{s_1+4-L}\|f\|_{L^{\infty}(\M)},
\end{align*}
where we used the mean value and Cauchy-Schwarz inequalities. This is again enough for \eqref{QS}.

Thus it remains to treat the contribution of
\begin{align*}
\Big[(\kk^\star\chi_0)-\sum_{k=0}^{K-1}t^{\frac{k}2}p_k(x,\sqrt{t}D)\Big]\sum_{\ell=0}^{L-1}M^{-\ell}q_\ell(x,M^{-1}D)\kk^\star\chi_2.
\end{align*}
Again, similarly to \cite[Lemma 2.6 (ii)]{ORTW}, we can express the composition of pseudo-differential operators
$p_k(x,\sqrt{t}D)q_\ell(x,M^{-1}D)$ as
\begin{align*}
p_k(x,\sqrt{t}D)q_\ell(x,M^{-1}D) = \sum_{|\nu|=0}^{B-1}\frac{t^{\frac{|\nu|}2}}{\nu!}\big\{\partial_\eta^\nu p_k(x,\sqrt{t}\cdot)\partial_x^\nu q_\ell(x,M^{-1}\cdot)\big\}(D)+R_{k,\ell,B,t,M}(x,D)
\end{align*}
where the symbol of the remainder is given by
\begin{align*}
R_{k,\ell,B,t,M}(x,\eta)&=\frac1{(2\pi)^2}\sum_{|\nu|=B}\frac{Bt^{\frac{B}2}}{\nu!}\int_{\R^2}\int_{\R^2}e^{-iz\cdot\zeta}\partial_\zeta^\nu p_k\big(x,\sqrt{t}(\eta+\zeta)\big)\\
&\qquad\times\int_0^1(1-\theta)^B\partial_x^\nu q_\ell(x+\theta z,M^{-1}\eta)d\theta dzd\zeta.\notag
\end{align*}
In particular its kernel $\int_{\R^2}e^{i(x-y)\cdot\eta}R_{k,\ell,B,t,M}(x,\eta)d\eta$ can be estimated using integrations by part in $\zeta$ along with the property $p_k,q_\ell\in\S^{-\infty}(U\times\R^2)$:
\begin{align*}
&\Big|\int_{\R^2}e^{i(x-y)\cdot\eta}\int_{\R^2}\int_{\R^2}e^{-iz\cdot\zeta}\partial_\zeta^\nu p_k\big(x,\sqrt{t}(\eta+\zeta)\big)\int_0^1(1-\theta)^B\partial_x^\nu q_\ell(x+\theta z,M^{-1}\eta)d\theta dzd\zeta d\eta\Big|\notag\\
&\les\int_{\R^2}\int_{\R^2}\int_{\R^2}\jb{z}^{-3}\Big|\jb{D_\zeta}^3\partial_\zeta^\nu p_k\big(x,\sqrt{t}(\eta+\zeta)\big)\int_0^1(1-\theta)^B\partial_x^\nu q_\ell(x+\theta z,M^{-1}\eta)d\theta\Big| dzd\zeta d\eta\Big|\\
&\les \int_{\R^2}\int_{\R^2}\jb{\sqrt{t}(\eta+\zeta)}^{-10}\mathbf{1}_{|\eta|\les M}d\zeta d\eta \les t^{\frac{B-2}2}M^2,\notag
\end{align*}
for any $|\nu|=B$, uniformly in $x$ and $y$. This leads to
\begin{align*}
\big\|R_{k,\ell,B,t,M}(x,D)\kk^\star\chi_2\big\|_{L^{\infty}(\M)\to L^{\infty}(\R^2)}&\les t^{\frac{B}2}\sup_{|\nu|=B}\Big\|\int_{\R^2}e^{i(x-y)\cdot\eta}R_{k,\ell,B,t,M}(x,\eta)d\eta\Big\|_{L^\infty_{x,y}}\\&\les t^{\frac{B-2}2}M^2,\notag
\end{align*}
where we used Schur's test in the first step, along with the compactness of the support in $x$ and $y$. This last bound is also enough for \eqref{QS} by taking $B$ large, since we are in the regime $\sqrt{t}\ll M^{-1}$.

As for the products $\partial_\eta^\nu p_k\partial_x^\nu q_\ell$, in the case $k+|\nu|\ge 1$, proceeding similarly and using integrations by part we have
\begin{align*}
&\Big\|t^{\frac{k+|\nu|}2}\big\{\partial_\eta^\nu p_k(x,\sqrt{t}\cdot)\partial_x^\nu q_\ell(x,M^{-1}\cdot)\big\}(D)(\kk^\star\chi_2)\Big\|_{L^{\infty}(\R^2)\to L^{\infty}(\R^2)}\\ &\les t^{\frac{k+|\nu|}2}\Big\|(\kk^\star\chi_2)(y)\int_{\R^2}e^{i(x-y)\cdot\eta}\partial_\eta^\nu p_k(x,\sqrt{t}\eta)\partial_x^\nu q_\ell(x,M^{-1}\eta)d\eta\Big\|_{L^\infty_x L^1_y}\\
&\les  t^{\frac{k+|\nu|}2}\Big\|\frac{(\kk^\star\chi_2)(y)}{|x-y|}\int_{\R^2}e^{i(x-y)\cdot\eta}\big\{\sqrt{t}\partial_\eta^\nu\nabla_\eta p_k(x,\sqrt{t}\eta)\partial_x^\nu q_\ell(x,M^{-1}\eta)\\
&\qquad\qquad+M^{-1}\partial_\eta^\nu p_k(x,\sqrt{t}\eta)\partial_x^\nu \nabla_\eta q_\ell(x,M^{-1}\eta)\big\}d\eta\Big\|_{L^\infty_x L^1_y}\\
&\les t^{\frac{k+|\nu|}2}\sup_x\int_{\R^2}\int_{\R^2}\frac{(\kk^\star\chi_2)(y)}{M|x-y|}\jb{\sqrt{t}\eta}^{-10}\mathbf{1}_{|\eta|\les M}d\eta dy\les t^{\frac{k+|\nu|}2}M,
\end{align*}
where we used that $\sqrt{t}\ll M^{-1}$. This is still enough for \eqref{QS} due to the restriction $s_2 - s_1 \le 1 \le k+|\nu|$.

At last, using that the symbol of the leading term corresponding to $k=|\nu|=0$ is given by 
\begin{align*}
(\kk^\star\chi_0)(x)\big[1-e^{-t(\alpha_1+i\alpha_2)|\eta|^2_x}\big]q_\ell(x,M^{-1}\eta)
\end{align*}
with $q_\ell$ compactly supported in $\eta$, a straightforward adaptation of the argument in \cite[Lemma 4]{MW2} yields
\begin{align*}
\Big\|(\kk^\star\chi_0)q_\ell(x,M^{-1}D)-\big\{p_0(x,\sqrt{t}\cdot)q_\ell(x,M^{-1}\cdot)\big\}(D)\Big\|_{L^{\infty}(\R^2)\to L^{\infty}(\R^2)}\les tM^2,
\end{align*}
which finishes the proof of \eqref{QS}.

Finally, let us prove $\textup{(iii)}$. For $t\in [0,\infty)$, let $\eps>0$. In view of Definition~\ref{Besov} for the space $\Cc^s(\M)$, we can pick $g\in C^{\infty}(\M)$ such that $\big\|f-g\big\|_{\Cc^s}\ll \eps$. Then, for small $h \in \R$ (or $h>0$ if $t=0$), using \textup{(i)} and \textup{(ii)} we estimate for some $\dl\in (0,1)$
\begin{align*}
\big\| \big( S_{A,\alpha}(t+h) - S_{A,\alpha}(t) \big)f \big\|_{\Cc^s} & \les \big\|\big( S_{A,\alpha}(t+h) - S_{A,\alpha}(t) \big)(f-g)  \big\|_{\Cc^s} +  \big\| S_{A,\alpha}(t)   \big( S(h) - \Id \big)g \big\|_{\Cc^s} \\
& \les  \big\| f- g \big\|_{\Cc^s} +  \big\| \big( S(h) - \Id \big)g \big\|_{\Cc^s} \\
&  \les \big\| f- g \big\|_{\Cc^s} +  h^{\frac{\dl}2} \| g \|_{\Cc^{s+ \dl}} < \eps
\end{align*}
provided that $h$ is small enough depending on $\eps$. This proves the property (iii).
\end{proof}

We also have the following linear and bilinear estimates in Besov spaces.
\begin{lemma} \label{LEM:Besov} We have the following inequalities: \\
\textup{(i)} Suppose $s_1 <0 $ and $s_2 >0$ satisfy $s_1+s_2 >0$. Then, the map $(f,g) \mapsto fg $ can be extended to a continuous linear mapping from $\Cc^{s_1}(\M) \times \Cc^{s_2}(\M) \to \Cc^{s_1}(\M)$ and
\begin{align*}
\| fg \|_{\Cc^{s_1}} \les \| f \|_{\Cc^{s_1}} \| g \|_{\Cc^{s_2}}.
\end{align*}
\textup{(ii)} For any $s >0$ and $f,g \in \Cc^{s}(\M)$ we have:
\begin{align*}
\| fg \|_{\Cc^{s}(\M)} \les \| f \|_{L^{\infty}} \| g \|_{\Cc^{s}} +  \| g \|_{L^{\infty}} \| f \|_{\Cc^{s}}.
\end{align*}
\textup{(iii)} Let $s \in\R$, $1\leq p_1\leq p_2\leq \infty$ and $q\in [1,\infty]$. Then for any $f\in \Bb^s_{p_1,q}(\M)$ we have
\begin{align*}
\|f\|_{\Bb^{s-2\big(\frac1{p_1}-\frac1{p_2}\big)}_{p_2,q}}\les \|f\|_{\Bb^{s}_{p_1,q}}.
\end{align*}
\textup{(iv)} Let $s>0$ and $p_1,p_2,q_1,q_2\in [1,\infty]$ with
\begin{align*}
\frac1p = \frac1{p_1}+\frac1{p_2} \qquad\text{ and }\qquad \frac1q=\frac1{q_1}+\frac1{q_2}.
\end{align*}
 Then for any $f\in B^{s}_{p_1,q_1}(\M)$ and $g\in B^{s}_{p_2,q_2}(\M)$, we have $fg \in B^{s}_{p,q}(\M)$, and moreover it holds
\begin{align*}
\|fg\|_{B^{s}_{p,q}(\M)}\les \|f\|_{B^{s}_{p_1,q_1}(\M)}\|g\|_{B^{0}_{p_2,q_2}(\M)}+\|f\|_{B^{0}_{p_1,q_1}(\M)}\|g\|_{B^{s}_{p_2,q_2}(\M)}.
\end{align*}
\textup{(v)} Let $\theta\in (0;1)$. There exists $C>0$ such that for any smooth $f$ it holds
\begin{align*}
\|f\|_{B^\theta_{1,1}}\le \|f\|_{L^1}^{1-\theta}\|\nabla f\|_{L^1}^\theta+\|f\|_{L^1}.
\end{align*}
\end{lemma}
\begin{proof}
By \cite[Proposition 2.5]{ORT}, the proof of these statements on a closed surface $(\M,\gm)$ reduces to the corresponding statements in the Euclidean case $\M = \R^2$, which can be found e.g. in \cite[Theorems 2.82, 2.85 and Corollary 2.86]{BCD} or \cite[Proposition 8]{MW2}.
\end{proof}
\subsection{Regularization of the Green's function}
We now investigate the properties of the fundamental solution for the magnetic Laplace equation, defined as\footnote{Recall that $\lambda_0=0$ if and only if $A$ is exact. For the local and global well-posedness in Theorems~\ref{THM : LWP theory 1} and~\ref{THM : deterministic GWP}, we only need to know the singularity of the Green's function along the diagonal, which does not depend on whether $\lambda_0=0$ or not. This explains why we allow $\lambda_0=0$ in the definition \eqref{Green} of the Green's function.}
\begin{align}\label{Green}
G_A(x,y) \deff \sum_{\ld_n\ge \ld_+}\frac{\varphi_n(x)\cj{\varphi_n(y)}}{\ld_n^2}, \quad (x,y) \in \M\times\M,
\end{align}
where $\lambda_+$ has been defined in \eqref{lambda+}, and the convergence of the series holds in the sense of distributions.

Following \cite[Subsection 2.2]{ORTW}, the first step is to show that $G_A$ has at most a logarithmic singularity near the diagonal.

\begin{lemma}\label{LEM:GreenLog}
\textup{(i)} There exists a continuous function $\widetilde{G_A}$ on $\M\times\M$ such that for any $x,y\in\M$ with $x\neq y$, it holds
\begin{align}\label{GAlog}
G_A(x,y)=-\frac1{2\pi}\log(\dg(x,y)) +\widetilde{G_A}(x,y).
\end{align}
\textup{(ii)} There exists a constant $C>0$ such that for any $x,y\in\M$ with $x\neq y$, it holds
\begin{align}\label{GAnabla}
\big|\nabla_x G_A(x,y)\big|\le C\dg(x,y)^{-1}.
\end{align}
\end{lemma}
\begin{proof}
Since $-\Dla$ has the same principal symbol as $-\Dlg$, we will recover the bounds of Lemma~\ref{LEM:GreenLog} from that of the standard Laplacian in $\R^2$. First, we can get \eqref{GAlog} by using the following resolvent identity:
\begin{align*}
(1-\Dla)^{-1}-(1-\Dlg)^{-1}=(1-\Dla)^{-1}(\Dlg-\Dla)(1-\Dlg)^{-1} : H^s(\M)\to H^{s+3}(\M)
\end{align*}
for any $s\in\R$, where we used that $\Dlg-\Dla$ is a differential operator of order 1; see \eqref{Delta}. This shows that the kernel $G_{\langle A\rangle}$ of $(1-\Dla)^{-1}$ differs from that of $(1-\Dlg)^{-1}$ by a continuous function. Moreover, denoting $G_{\langle 0\rangle}$ for the massive Green's function, which is the kernel of $(1-\Dlg)^{-1}$, it holds
\begin{align*}
G_A(x,y)-G_{\langle A\rangle}(x,y)&= \sum_{\ld_n\ge \ld_+}\big(\frac1{\ld_n^2}-\frac1{\jb{\ld_n}^2}\big)\varphi_n(x)\cj{\varphi_n(y)} + \mathbf{1}_{\lambda_0=0}\varphi_0(x)\cj{\varphi_0(y)}\\
&=\sum_{\ld_n\ge\ld_+}\frac1{\ld_n^2\langle\ld_n\rangle^2}\varphi_n(x)\cj{\varphi_n(y)}+ \mathbf{1}_{\lambda_0=0}\varphi_0(x)\cj{\varphi_0(y)},
\end{align*}
where the right-hand side is a continuous function in view of Corollary~\ref{COR:avg eig} and \eqref{asympt ldn}. Similarly, $G_{\langle 0\rangle}-G_0$ is a continuous function, where $G_0$ is the usual Green's function of the Laplace-Beltrami operator on $(\M,\gm)$. It is well-known (see e.g. \cite[Section 4.2]{Aubin}) that $G_0$ satisfies \eqref{GAlog}, which, from the previous discussion, shows that $G_A$ does too.

To get \eqref{GAnabla}, we express
\begin{align}\label{GA}
G_A=\int_0^1P_A(t)dt+R_A
\end{align}
where $$P_A(t,x,y)=\sum_{n\ge 0}e^{-t\ld_n^2}\varphi_n(x)\cj{\varphi_n(y)}$$ is the kernel of $e^{t\Dla}$, and 
\begin{align*}
R_A(x,y)=\sum_{\ld_n\ge\ld_+}\frac{e^{-\ld_n^2}}{\ld_n^2}\varphi_n(x)\cj{\varphi_n(y)}-\mathbf{1}_{\ld_0=0}\varphi_0(x)\cj{\varphi_0(y)}.
\end{align*}
In particular $R_A$ is smooth on $\M\times\M$, so that we only need to estimate the first term in \eqref{GA}. Since $P_A(t,x,y)$ is smooth for any $t\ge 0$ if $\dg(x,y)$ is large enough, we only estimate the case where $\dg(x,y)\ll 1$. Now we notice that we can express, through the functional calculus of $-\Dla$, $e^{t\Dla}$ as $\psi(-t\Dla)$ where $\psi$ is a Schwartz function such that $\psi(z)=e^{-z}$ for $z\ge0$. Using then the notations from \cite{BGT,ORT,ORTW} as in the proof of Lemma~\ref{LEM:semigroup} above, we have in geodesic normal coordinates around $x$
\begin{align*}
P_A(t,x,y)=\sum_{k=0}^{K-1}t^\frac{k}2\int_{\R^2}e^{i(x-y)\cdot\eta}p_k(x,\sqrt{t}\eta)d\eta + R_{-K,t}(x,y),
\end{align*}
for arbitrary $K\ge 0$. From the smoothing property \eqref{remainder} of the operator with kernel $R_{-K,t}$ and Sobolev embedding, we get
\begin{align*}
\|R_{-K,t}\|_{L^\infty_{x,y}}\les \|R_{-K,t}\|_{H^{-1-\delta}\to H^{1+\delta}} \les t^\frac{K-2-2\delta}2
\end{align*}
for any $0<\dl\ll 1$. Similarly,
\begin{align*}
\|\nabla_x R_{-K,t}\|_{L^\infty_{x,y}}\les \|R_{-K,t}\|_{H^{-1-\delta}\to H^{2+\delta}}\les t^\frac{K-3-2\delta}2.
\end{align*}
This is enough for \eqref{GAnabla} after integrating in time, provided that we take $K\ge 2$.

As for the kernels in the sum for $k\ge 1$, since $p_k$ is Schwartz in $\eta$, after changing variables we get
\begin{align*}
t^\frac{k}{2}\Big\|\int_{\R^2}e^{i(x-y)\cdot\eta}p_k(x,\sqrt{t}\eta)d\eta\Big\|_{L^\infty_{x,y}}&=t^\frac{k-2}{2}\Big\|\int_{\R^2}e^{i(x-y)\cdot t^{-\frac12}\zeta}p_k(x,\zeta)d\zeta\Big\|_{L^\infty_{x,y}}\\
&\les t^\frac{k-2}2\|p_k\|_{L^\infty_xL^1_\eta},
\end{align*}
and integrating by parts,
\begin{align*}
&t^{\frac{k-2}{2}}\Big\|\nabla_x\int_{\R^2}e^{i(x-y)\cdot t^{-\frac12}\zeta}p_k(x,\zeta)d\zeta\Big\|_{L^\infty_{x,y}}\\
&\qquad\qquad= t^{\frac{k-2}2}\Big\|\int_{\R^2}e^{i(x-y)\cdot t^{-\frac12}\zeta}\big(i t^{-\frac12}p_k(x,\zeta)\zeta+\nabla_x p_k(x,\zeta)\big)\Big\|_{L^\infty_{x,y}}\\
&\qquad\qquad\les t^{\frac{k-2}{2}}\Big\|\int_{\R^2}\big(|\nabla_\zeta p_k(x,\zeta)|t^{-\frac12}\frac{\sqrt{t}}{|x-y|}|\zeta|+|\nabla_x p_k(x,\zeta)|\big)\Big\|_{L^\infty_{x,y}}\\
&\qquad\qquad\les \frac{t^{\frac{k-2}2}}{|x-y|}\big(\big\||\zeta||\nabla_\zeta p_k|\big\|_{L^\infty_x L^1_\zeta}+\|\nabla_xp_k\|_{L^\infty_xL^1_\zeta}\big).
\end{align*}
This is again enough when $k\ge 1$ to get \eqref{GAnabla} after integrating in time.

Finally, since the principal symbol in the expansion above is given by $p_0(x,\sqrt{t}\eta)=\psi(t|\eta|_x^2)=e^{-t|\eta|_x^2}$, the contribution of the principal part is
\begin{align*}
&\gm(x)^{j,\ell}\partial_{x_j}\int_0^1\int_{\R^2}e^{i(x-y)\cdot \eta}e^{-t\gm(x)\eta\cdot\eta}d\eta dt\\
&= \gm(x)^{j,\ell}\int_0^1\partial_{x_j}\big[|\gm(x)|^{-\frac12}(2\pi t)^{-1}e^{-\frac1{4t}\gm(x)^{-1}(x-y)\cdot(x-y)}\big]dt\\
&=\gm(x)^{j,\ell}\int_0^1\big[\partial_{x_j}|\gm(x)|^{-\frac12}-\frac1{4t}(\partial_{x_j}\gm(x)^{j_1,j_2})(x_{j_1}-y_{j_1})(x_{j_2}-y_{j_2}) - \frac{\gm(x)^{j,j_2}(x_j-y_j)}{2t}\big]\\
&\qquad\qquad\qquad\qquad\times(2\pi t)^{-1} e^{-\frac1{4t}\gm(x)^{-1}(x-y)\cdot(x-y)}dt.
\end{align*}
Now, since we use geodesic normal coordinates around $x$, we have $\gm(x)_{j_1,j_2}=\gm(x)^{j_1,j_2}=\delta_{j_1,j_2}$ and $\partial_j\gm(x)^{j_1,j_2}=0$. This yields
\begin{align*}
\Big\|\nabla_x\Big(\int_0^1\int_{\R^2}e^{i(x-y)\cdot\eta}p_0(x,\sqrt{t}\eta)d\eta dt\Big)\Big\|_{L^\infty_{x,y}}&\les |x-y| \int_0^1 t^{-2}e^{-\frac{|x-y|^2}{4t}}dt \intertext{Making the change of variables $r=\frac{|x-y|}{2\sqrt{t}}$, we can continue with}
&\les |x-y| \int_{|x-y|}^\infty \frac{r^4}{|x-y|^4}e^{-r^2}\frac{|x-y|^2}{r^3}dr\\
&\les |x-y|^{-1}\Big[-e^{r^2}\Big]_{|x-y|}^\infty \les |x-y|^{-1}.
\end{align*}
This finally shows \eqref{GAnabla} and finishes the proof of Lemma~\ref{LEM:GreenLog}.
\end{proof}

Next, we investigate pointwise bounds on the regularized Green's function that we will need in constructing the stochastic objects in Section~\ref{SEC : proba}. Hereafter, we write $\diag \deff \{(x,x),~x\in\M\}$ for the diagonal in $\M\times\M$.
\begin{lemma}\label{LEM:green function 1}
Let $\psi\in\S(\R)$ such that $\psi(0)=1$. Then:\\
\textup{(i)} There exists $C>0$ such that for any $N\in\N$ and $(x,y)\in\M\times\M\setminus\diag$ it holds
\begin{align}\label{GN1}
\Big|(\psi\otimes\psi)\big(-N^{-2}\Dla\big)G_A(x,y) +\frac1{2\pi}\log\big(\dg(x,y)+N^{-1}\big)\Big| \le C.
\end{align}
\textup{(ii)} Let $N_1, N_2 \in \N$ and fix $j \in \{1,2\}$. Then for all $0<\dl\ll1$ there exists $C>0$ such that for any $(x,y)\in\M\times\M\setminus\diag$ it holds that 
\begin{align}\label{GN2}
\begin{split}
&\Big|(\psi \otimes\psi)\big(-N_j^{-2}\Dla\big)G_A(x,y)-\psi\big(-N_1^{-2}\Dla\big)  \otimes\psi \big(-N_2^{-2}\Dla \big) G_A(x,y)\Big|\\
&\qquad\le C\min\Big\{-\log\big(\dg(x,y)+\min(N_1,N_2))^{-1}\big)\vee 1; \min(N_1,N_2)^{\dl-1}\dg(x,y)^{-1}\Big\}.
\end{split}
\end{align}
\end{lemma}
\begin{proof}
The bounds \eqref{GN1} and \eqref{GN2} with the Green function of $-\Dlg$ in place of $G_A$ follow from \cite[Lemma 2.8, Corollary 2.12, and Lemma 2.13]{ORTW}. Since they only relied on the estimates of Lemma~\ref{LEM:GreenLog} and on the expression of the principal symbol of $-\Dlg$, which is the same for $-\Dla$, the same proofs apply here.
\end{proof}

\section{Probabilistic tools}\label{SEC : proba}
We now recall the probabilistic tools that we will use when estimating the renormalized stochastic objects $:\!\Psi_N^{j_1}\cj{\Psi_N}^{j_2}\!:$ appearing in \eqref{renormalization stochastic objects}.
\subsection{Algebraic properties of the renormalization} 
We start by looking at the algebraic properties of the renormalization performed in \eqref{renormalization nonlinearity}. We refer to \cite{Trenberth} for the proofs of these identities. The first one is concerned with the expression of generalized Laguerre polynomials of a sum.
\begin{lemma}\label{LEM:sum}
Let $m \geq 1$ and $l \geq 0$. Then the following identity holds:
\begin{align*}
(-1)^m m! L_m^{(1)} \big( |x+ y|^2; \s \big)(x+y) = \sum_{  \substack{ 0 \leq j_1 \leq m+1 \\ 0 \leq j_2 \leq m } } \binom{m+1}{j_1} \binom{m}{j_2} P^{m,\s}_{j_1,j_2}(y, \cj{y}) x^{j_1} \cj{y}^{j_2}, \quad (x,y) \in \C^2,
\end{align*}
where
\begin{align*}
P^{m,\s}_{j_1,j_2}(y, \cj{y}) = \begin{cases} (-1)^{m-j_2} (m-j_2)! L^{(j_2-j_1+ 1)}_{m-j_2} \big( |y|^2; \s) y^{j_2-j_1+1},& \quad j_1+1 \geq j_1 \\ (-1)^{m+1-j_1} (m+1-j_1)! L^{(j_1-j_2 -1)}_{m+1-j_1} \big( |y|^2; \s)  \cj{y}^{j_1-j_2-1},& \quad j_2+1 < j_1
\end{cases}
\end{align*}
\end{lemma}

The second one deals with the covariance of generalized Laguerre polynomials of Gaussian random variables.
\begin{lemma} \label{LEM:covariance prop} 
Let $f$ and $g$ be mean-zero complex \textup{(}jointly\textup{)} Gaussian random variables with variances $\s_f$ and $\s_g$ respectively. Then,
\begin{align*}
\E \left[  L_k^{(l)} \big( |f|^2; \s_f \big) f^l \cj{ L_m^{(l)} \big( |g|^2; \s_g \big) g^l }  \right] = \delta_{km} \frac{(k+l)!}{k!}   |\E[ f \cj{g} ]|^{2k}  \E[ f \cj{g} ] ^l .
\end{align*}
\end{lemma}

This latter result will prove particularly useful when combined with the Wiener chaos estimate, which we recall in the next lemma (the proof of this last statement can be found for example in \cite[Theorem I.22]{Simon}).
\begin{lemma}\label{LEM:wiener}
 Let $ g = \{ g_n \}_{n \in \Z}$ be a sequence of standard independent and identically distributed real Gaussian random variables. Let $k \in \N$ and let $ \{ P_j ( g) \}_{j \in \N}$ be a sequence of polynomials of degree at most $k$. Then for $ p \geq 2$,
\begin{equation*}
\big\| \sum_{j \in \N} P_j ( g )  \big\|_{L^p(\O)} \leq (p-1)^{\frac{k}{2}} \big\| \sum_{j \in \N} P_j ( g )  \big\|_{L^2(\O)}.
\end{equation*} 
\end{lemma}
\subsection{On the Wick powers}
We now move on to the construction and the study of the stochastic objects $\Psi$ and $:\!\Psi_N^{j_1}\cj{\Psi_N}^{j_2}\!:$ defined in \eqref{stochastic convolution} and \eqref{renormalization stochastic objects} respectively.
\begin{proposition}\label{PROP:construction stochastic objects}
Fix $j_1,j_2 \geq 0$ two integers. For any $0 < \e \ll 1$,  $T >0$ and $1 \leq p,q < \infty$, the random variables $\{ :\!\PN^{j_1}\cj{\PN}^{j_2}\!:\}_{N \in \N}$ form a Cauchy sequence in $L^p \left( \O; L^q([0,T]; \Cc^{-\e}(\M)) \right)$ and converge almost surely to a limit $:\!\Psi^{j_1}\cj{\Psi}^{j_2}\!: \,\in L^q([0,T]; \Cc^{-\e}(\M)) $. Moreover, we have the following tail estimate: there exists some constants $c,C>0$ and $0<\dl\ll \eps$ such that for any $T,R>0$ and any\footnote{with the convention that $\,:\!\Psi_N^{j_1}\cj{\Psi_N}^{j_2}\!:\,=\,:\!\Psi^{j_1}\cj{\Psi}^{j_2}\!:\,$ when $N=\infty$.} $1\le N_1\le N_2\le \infty$ it holds
\begin{align}
\Prob \Big(  \big\| \,:\!\Psi_{N_1}^{j_1} \cj{\Psi_{N_1}}^{j_2}\!:\,  -  \,:\!\Psi_{N_2}^{j_1} \cj{\Psi_{N_2}}^{j_2}\!:\,  \big\|_{L^q_T \Cc^{-\e}}   > R \Big) \leq C e^{-c R^{\frac{2}{j_1+j_2}} N_1^{\frac{\dl}{j_1+j_2}} T^{- \frac{2}{q(j_1+j_2)}}\jb{T}^{-1}},
\label{deviation}
\end{align}
 and
\begin{align}
\Prob \Big(  \big\| :\!\PN^{j_1}\cj{\PN}^{j_2}\!:\big\|_{L^q_T \Cc^{-\e}}   > R \Big) \leq C e^{-c R^{\frac{2}{j_1+j_2}} T^{- \frac{2}{q(j_1+j_2)}}\jb{T}^{-1}}.
\label{deviation2}
\end{align}
 At last, we also have $\Psi \in C\left( [0,T]; \Cc^{-\e}(\M) \right)$ almost surely, and
 \begin{align}
 \Prob \Big(  \big\| \PN - \Psi \big\|_{C_T\Cc^{-\eps}}  > R \Big) \leq C e^{-c R^2 N^{\dl}}.
\label{deviation3}
 \end{align}
\end{proposition}
\begin{proof}Fix $j_1,j_2 \geq 0$, $0 < \e \ll 1$ and $T >0$. By H\"older's inequality, it suffices to prove that $\{ :\!\PN^{j_1}\cj{\PN}^{j_2}\!:\}_{N \in \N}$ forms a Cauchy sequence in $L^{p} \left( \O; L^q([0,T]; \Cc^{-\e}(\M)) \right)$ when $p\ge q\gg 1$. 

\medskip

\noi
{\it \underline{$\bul$ Step 1: uniform boundedness.}} We start by showing that $\displaystyle{\{\,:\!\PN^{j_1}\cj{\PN}^{j_2}\!:\,\}_{N\ge 1}}$ is uniformly bounded in $L^{p} \left( \O; L^q([0,T]; \Cc^{-\e}(\M)) \right) $. First, a straightforward computation using Ito's isometry and the definition \eqref{stochastic convolution} of the stochastic convolution shows that for any $N_1,N_2\ge 1$, any $y,z \in \M$ and $t\in[0;T]$, the covariance function is given by
\begin{align}
\begin{split}\label{OBJ1}
\mathfrak G_{N_1,N_2}(t,y,z)&\deff\E[ \Psi_{N_1}(t,y) \cj{\Psi_{N_2}} (t,z) ] \\
&= \frac{\g}{\alpha_1}\big( \psi(- N_1^2 \Dl) \otimes \psi( -N_2^2 \Dl) \big) G_A(y,z)+2\mathbf{1}_{\ld_0=0}\gamma t\varphi_0(y)\cj{\varphi_0(z)},
\end{split}
\end{align}
where the Green function $G_A$ has been defined in \eqref{Green} above. Then, by a standard combination (see e.g. \cite{GKO1,OT2,ORT}) of Sobolev inequality (Lemma~\ref{LEM:Besov} (iii)), Minkowski's inequality, and the Wiener chaos estimate  (Lemma~\ref{LEM:wiener}), along with the definition~\eqref{Besov norm} of Besov spaces, we have the following estimates for some $r_{\e}\gg 1$ and $ p \geq \max(q, r_{\e}) $:
\begin{align}
\begin{split}\label{OBJ2}
\big\| :\!\PN^{j_1}\cj{\PN}^{j_2}\!:\big\|_{L^{p}_{\O} L^q_T \Cc^{-\e}} & \les  \big\| :\!\PN^{j_1}\cj{\PN}^{j_2}\!:  \big\|_{L^{p}_{\O}L^{q}_{T}  B^{-\frac{\e}{2}}_{r_{\e}, r_{\e}}} \\ 
& \les \big\| \jb{M}^{-\frac{\e}{2}} \Q_M \big[:\!\PN^{j_1}\cj{\PN}^{j_2}\!:\big]\big\|_{L^{q}_{T} \l_M ^{ r_{\e} } L_x^{r_{\e}}L^{p}_{\O}}   \\
& \les p^{\frac{j_1+j_2}2} \big\|\jb{M}^{-\frac{\e}{2}} \Q_M \big[:\!\PN^{j_1}\cj{\PN}^{j_2}\!:\big]\big\|_{L^{q}_{T}  L_x^{r_{\e}} \l_M ^{ r_{\e} } L^{2}_{\O}}.
\end{split}
\end{align}
To estimate the inner $L^2(\O)$ norm, noting $\K_M$ the kernel of $\Q_M$ as in Lemma~\ref{LEM:PM}, we use Lemma \ref{LEM:covariance prop} with the definition \eqref{renormalization stochastic objects} of the renormalized stochastic process $\displaystyle{\{\,:\!\PN^{j_1}\cj{\PN}^{j_2}\!:\,\}_{N\ge 1}}$ , the expression \eqref{OBJ1} for the covariance function of $\Psi_N$ along with the pointwise bounds \eqref{K2} and \eqref{GN1} on $\K_M$ and on the regularized Green's function to estimate for fixed $M \in 2^{\Z_{-1}}$, $x\in\M$ and $t\in[0;T]$:
\begin{align*}
& \E \left[ M^{-\frac{\e}4} | \Q_M :\PN^{j_1} \cj{\PN}^{j_2}:(t,x) |^2 \right]  \\
& = M ^{-\frac{\e}4}  \int_\M\int_\M \mathcal{K}_M (x,y) \mathcal{K}_M(x,z) \E\bigg[:\!\PN^{j_1} \cj{\PN}^{j_2}\!:\,(t,y) \cj{\,:\!\PN^{j_1} \cj{\PN}^{j_2}\!:\,(t,z)}\bigg]  dV(y) dV(z) \\
& =  C  M ^{-\frac{\e}4} \int_\M\int_\M \mathcal{K}_M (x,y) \mathcal{K}_M(x,z) \ \E\Big[ \PN(t,y) \cj{  \PN (t,z) } \Big]^{j_1} \  \cj{\E\Big[ \PN(t,y) \cj{  \PN (t,z) } \Big]}^{j_2} dV(y) dV(z) \\
& \les_L M ^{-\frac{\e}4} \int_\M\int_\M  M^2 \jb{ M \dg(x,y)  }^{-L} \ M^2 \jb{ M \dg(x,z)  }^{-L}\\
&\qquad\qquad\qquad\quad\times \big[\  \langle\log(\dg(y,z)+N^{-1})\rangle^{j_1+j_2}+T^{j_1+j_2}\big] dV(y) dV(z)\intertext{for some constant $C=C(j_1,j_2)$ and for arbitrary $L \gg 1$. Using the inequalities $ | \log(r) | \les_\al r^{-\delta}$ for $0<r \les 1$ and $M^2 \jb{ M \dg(x,y)  }^{-L}\les_\be M^{2\dl} \dg(x,y)^{2\dl - 2} $ for arbitrary $0<\delta\ll\eps$, we can continue with}
& \les M ^{-\frac{\e}4+4\dl} \int_\M\int_\M  \dg(x,y)^{2\dl-2}  \dg(x,z)^{2\dl-2} \big[\dg(y,z)^{-\dl} +T^{j_1+j_2}\big]dV(y) dV(z)\\
&\les M^{- \frac{\e}{4}+4\dl}\jb{T}^{j_1+j_2},
\end{align*}
uniformly in $M\in 2^{\Z_{\ge -1}}$, $x\in\M$ and $t\in[0;T]$, where the last step follows from classical convolution type inequalities; see \cite[Proposition 4.12]{Aubin}. Plugging this estimate into \eqref{OBJ2} and using H\"older's inequality with the compactness of $\M$, and choosing $0<\dl<\frac\eps4$, we get
\begin{align*}
\big\| :\!\PN^{j_1}\cj{\PN}^{j_2}\!:\big\|_{L^{p}_{\O} L^q_T \Cc^{-\e}} \les p^{\frac{j_1+j_2}2}T^{\frac1q}\jb{T}^{\frac{j_1+j_2}{2}},
\end{align*}
uniformly in $N\ge 1$.

\medskip

\noi
{\it \underline{$\bul$ Step 2: difference estimate.}} Let us now show that ${\displaystyle  \{ :\!\PN^{j_1}\cj{\PN}^{j_2}\!:\}_{N \in \N}}$ forms a Cauchy sequence in $L^{p} \left( \O; L^q([0,T]; C^{-\e, \infty}(\M)) \right)$. For the sake of concreteness, let us assume that $j_1\ge j_2$, the case $j_1<j_2$ being completely analogous. Using the definition \eqref{renormalization stochastic objects} of the renormalized process along with the algebraic property of the renormalization given in Lemma~\ref{LEM:covariance prop} and the expression \eqref{OBJ1} for the covariance function, we get for any $y,z\in\M$, any $t\in[0;T]$ and any $N_2\ge N_1\ge 1$:
\begin{align*}
& \E \left[  \left( :\!\Psi_{N_1}^{j_1} \cj{\Psi_{N_1}}^{j_2}\!:(t,y) \ - :\!\Psi_{N_2}^{j_1} \cj{\Psi_{N_2}}^{j_2}\!:(t,y)  \right) \cj{ \left( :\!\Psi_{N_1}^{j_1} \cj{\Psi_{N_1}}^{j_2}\!:(t,z) \ - :\!\Psi_{N_2}^{j_1} \cj{\Psi_{N_2}}^{j_2}\!:(t,z) \right) }     \right] \\
&\qquad = C\big|\mathfrak G_{N_1,N_1}(t,y,z)\big|^{2j_2} \mathfrak G_{N_1,N_1}(t,y,z)^{j_1-j_2} - C\big|\mathfrak G_{N_1,N_2}(t,y,z)\big|^{2j_2}\mathfrak G_{N_1,N_2}(t,y,z)^{j_1-j_2} \\
&\qquad\qquad- C\big|\mathfrak G_{N_2,N_1}(t,y,z)\big|^{2j_2}\mathfrak G_{N_2,N_1}(t,y,z)^{j_1-j_2} +C\big|\mathfrak G_{N_2,N_2}(t,y,z)\big|^{2j_2}\mathfrak G_{N_2,N_2}(t,y,z)^{j_1-j_2}\\
 &\qquad= \1 - \II-\III+\IV,
\end{align*}
for some constant $C=C(j_1,j_2)$. Using the mean value inequality with similar estimates as above, we can bound for $0 < \dl \ll \e $:
\begin{align*}
\big| \1 - \II\big| & \les\big|\mathfrak G_{N_1,N_1}(t,y,z)-\mathfrak G_{N_1,N_2}(t,y,z)\big|\big[ \dg(y,z)^{-\dl}+\jb T^{j_1 + j_2}\big].\intertext{Using then \eqref{GN2} with the definition \eqref{OBJ1} of $G_{N_1,N_2}$, we can continue with}
& \les\min\Big\{-\log\big(\dg(x,y)+N_1^{-1}\big)\vee 1; N_1^{\dl-1}\dg(x,y)^{-1}\Big\}\big[\dg(y,z)^{-\dl}+\jb T^{j_1 + j_2}\big] \\
& \les  \min\Big\{ \dg(x,y)^{- \dl};  N_1^{\dl-1}\dg(x,y)^{-1} \Big\}\big[\dg(y,z)^{-\dl}+\jb T^{j_1 + j_2}\big] \\
& \les N_1^{-\dl} \dg(x,y)^{-2\dl}\big[\dg(y,z)^{-\dl}+\jb T^{j_1 + j_2}\big].
\end{align*}
Similar considerations lead to
\[ \big| \III - \IV \big| \les N_1^{-\dl} \dg(x,y)^{-2\dl}\big[\dg(y,z)^{-\dl}+\jb T^{j_1 + j_2}\big]. \]
Thus, arguing as before we can estimate 
\[ \E \Big[ M^{-\e} \Big| \Q_M \big(:\Psi_{N_1}^{j_1} \cj{\Psi_{N_1}}^{j_2}:(t,x) \ - : \Psi_{N_2}^{j_1} \cj{ \Psi_{N_2} }^{j_2}:(t,x)\big)  \Big|^2 \Big] \les N_1^{-\dl}\jb T^{j_1 + j_2}, \]
uniformly in $t\in[0;T] $ and $x\in\M$, from which we obtain
\begin{align}\label{obj3}
\big\| :\!\Psi_{N_1}^{j_1} \cj{\Psi_{N_1}}^{j_2}\!: - :\!\Psi_{N_2}^{j_1} \cj{ \Psi_{N_2} }^{j_2}\!: \big\|_{L^{p}_{\O} L^q_T \Cc^{-\e}} \les p^{\frac{j_1+j_2}2} T^{\frac{1}{q}}\jb{T}^{\frac{j_1+j_2}2}  N_1^{-\frac{\dl}2}
\end{align}
for some small $0< \dl \ll \eps$ and $p \gg 1$. This proves that ${\displaystyle \{:\!\PN^{j_1} \cj{\PN}^{j_2}\!:\}_{N \in \N}}$ forms a Cauchy sequence in $L^{p} \big( \O; L^q([0,T]; \Cc^{-\e}(\M)) \big)$.

\medskip

\noi
{\it \underline{$\bul$ Step 3: tail estimates.}} Next, let us discuss the deviation bounds \eqref{deviation} and \eqref{deviation2}. Since \eqref{obj3} holds for any finite $p\ge 1$ by H\"older's inequality, we can use Chebychev's inequality so that for any $R>0$ and $N_2 \ge N_1$, we can bound
\begin{align}\label{obj4}
 \Prob \Big(  \big\| :\!\Psi_{N_1}^{j_1} \cj{\Psi_{N_1}}^{j_2}\!:  -  :\!\Psi_{N_2}^{j_1}\cj{\Psi_{N_2}}^{j_2}\!:\big\|_{L^q_T \Cc^{-\e}}   > R \Big) \le C R^{-p}p^{p\frac{j_1+j_2}2} T^{\frac{p}{q}}\jb{T}^{p\frac{j_1+j_2}{2}} N_1^{- \frac{\dl}2 p}.
\end{align} 
 Thus, optimizing \eqref{obj4} in $p$ yields \eqref{deviation}. The deviation bound \eqref{deviation2} is obtained by the same argument. Besides, by \eqref{deviation} and since the $L^p$-convergence implies the convergence in probability, we have that
\[ \Prob \Big(  \big\| :\!\Psi^{j_1}\cj{\Psi}^{j_2}\!: -  :\!\PN^{j_1}\cj{\PN}^{j_2}\!:\big\|_{L^q_T \Cc^{-\e}}   > R \Big) \leq C e^{-c R^{\frac{2}{j_1+j_2}} N^{\frac{\dl}{j_1+j_2}} T^{- \frac{2}{q(j_1+j_2)}}}, \]
for any $N \geq 1 $. Hence, by the Borel-Cantelli's Lemma, $ {\displaystyle :\!\PN^{j_1} \cj{\Psi_N}^{j_2}\!:} $ converges almost surely to $  :\!\Psi^{j_1}\cj{\Psi}^{j_2}\!:$ as $N$ goes to infinity. 

\medskip

At last, we show that $\Psi$ is continuous in time and prove \eqref{deviation3}. Fix $t\geq 0$, $h \in [-1,1]$, $x \in \M$ and $M \in 2^{\Z _{-1}}$. By independence of the increments of the Brownian motions and Ito's isometry, we have for $M\ge 1$
\begin{align*}
& \E \Big[ M^{-\e}  \big| \Q_M \big( \Psi(t+h,x) - \Psi(t,x) \big) \big|^2  \Big] \\
& \les M^{- \e} \sum_{\ld_n^2 \sim M^2} \frac{|\alpha_1+i\alpha_2|~| \varphi_n(x)|^2}{\ld_n^2} \big|  1-e^{- h(\alpha_1 + i\alpha_2) \ld_n^2}\big|. \intertext{Using the mean value inequality (recall that $\alpha_1>0$) with Corollary \ref{COR:avg eig}, the previous term can be bounded for some small $0 < \dl \ll \e$ by}
& \les M^{-\e} \sum_{\ld_n \sim M} \frac{ |\varphi_n(x)|^2}{\ld_n^{2} }   |h|^{\dl} \ld_n^{2 \dl} \les M^{-\e}|h|^{\dl} \sum_{\ld_n \sim M} \ld_n^{2\dl-2}  \\
& \les M^{2\dl-\frac\eps2} |h|^{\dl} \sum_{\ld_n \sim M} \ld_n^{-2-\frac\eps2} \les M^{2\dl-\frac\eps2}|h|^{\dl},
\end{align*}
uniformly in $t\in [0;T]$, $h\in[-t,T-t]$, $x\in\M$, and $M\in 2^{\Z_{\ge -1}}$, and where in the last step we used Weyl's law \eqref{asympt ldn}. We also have
\begin{align*}
\E \Big[ \big| \Q_0 \big( \Psi(t+h,x) - \Psi(t,x) \big) \big|^2  \Big]\les |h|.
\end{align*} 
 Arguing as before, this leads to
\begin{align*}
\big\|  \Psi(t+h) - \Psi(t)  \big\|_{L^p_{\O}\Cc^{-\e}}^p \les |h|^{\frac{p\dl}{2}}.
\end{align*}
Therefore choosing $p$ large enough so that $\frac{p\dl}{2} > 1$, we can apply Kolmogorov's continuity criterion (see e.g. \cite[Theorem 3.3]{DPZ}), which shows that $\Psi$ is almost surely continuous in time. The tail estimate \eqref{deviation3} then follows from a standard martingale estimate (see e.g. \cite[Theorem 3.9]{DPZ}) along with the same argument as above:
\begin{align*}
 \Prob \Big(  \big\| \PN - \Psi \big\|_{C_T\Cc^{-\eps}}  > R \Big) &\le R^{-p}\sup_{t\in[0,T]}\E\big\|\PN(t) - \Psi(t) \big\|_{\Cc^{-\eps}}^p\\
 &\les p^{\frac{p}2}R^{-p}\sup_{t\in[0,T]}\E\big\|\PN(t) - \Psi(t) \big\|_{\Cc^{-\eps}}^2 \les p^{\frac{p}2}R^{-p}N^{-\frac{\dl}2}.
\end{align*}
This concludes the proof of Proposition~\ref{PROP:construction stochastic objects}.
\end{proof}

\section{Well-posedness theory}\label{SEC : well-posedness }
	In this section, we study (the renormalized version of) the dynamical problem \eqref{SCGL 1} and establish Theorems~\ref{THM : LWP theory 1} and~\ref{THM : deterministic GWP}.
\subsection{Local well-posedness}
We begin by deriving a priori bounds on the solutions to the truncated renormalized equation \eqref{renormalized truncated SCGL 1}. Note that for any fixed $N\ge 1$, there is indeed a unique global solution $u_N$ to \eqref{renormalized truncated SCGL 1}, e.g. by using a standard energy method. The issue is then to derive suitable bounds to prove its convergence. Recall that we look for a solution $u_N$ to \eqref{renormalized truncated SCGL 1} under the form $u_N = v_N + \PN$, with $\PN $ being the stochastic convolution as in \eqref{truncated stochastic convolution}. Then, in view of the definition \eqref{renormalization nonlinearity} of the renormalized nonlinearity and of Lemma~\ref{LEM:sum}, the nonlinear remainder $v_N$ solves
\begin{align}
\begin{split}
 \dt v_N &= (\alpha_1 + i\alpha_2) \Dla v_N \\
 &\qquad- (\beta_1 + i \beta_2) \sum_{\substack{ 0 \leq j_1 \leq m \\  0 \leq j_2 \leq m-1}} \binom{m}{j_1} \binom{m-1}{j_2} v_N^{j_1} \cj{v_N}^{j_2} : \PN^{m-j_1} \cj{\PN}^{m-1-j_2}:
 \end{split}
 \label{vN}
\end{align}
with initial data $v_{N|t=0}= \P_N u_0 - \PN(0)$. From Proposition \ref{PROP:construction stochastic objects} we know that the processes ${\displaystyle  :\!\PN^{j_1}\cj{\PN}^{j_2}\!:}$ converge almost surely in $L^q([0,T]; \Cc^{s_0}(\M)) $ as $N\to\infty$, for any $T>0$, $q\ge 1$ and $s_0<0$.
This discussion leads to study the model equation 
\begin{equation}\label{model equation}
\begin{cases}
\dt v = (\alpha_1 + i\alpha_2) \Dla v - F(v, \Xi) \\
v_{|t=0}= v_0 \in \Cc^{s_0}(\M),	
\end{cases}
\end{equation}
where $\Xi=\{\Xi_{j_1,j_2}\}_{0\le j_1\le m, 0\le j_2 \le m-1}$ (with $\Xi_{m,m-1} =1$) is a collection of space-time distributions in $L^q([0,T]; \Cc^{s_0}(\M)) $, and the nonlinearity is given by
\begin{align}\label{def F}
F(v, \Xi) = (\beta_1 + i \beta_2) \sum_{\substack{ 0 \leq j_1 \leq m \\  0 \leq j_2 \leq m-1}} \binom{m}{j_1} \binom{m-1}{j_2} v^{j_1} \cj{v}^{j_2}  \Xi_{m-j_1,m-j_2-1}.
\end{align}

For $T \in (0,\infty)$, $\e > 0$ and $q \geq 1$, we first define the space $\mathcal Z^{\eps,q}([0,T])$ by
\begin{align}
\mathcal Z^{\eps,q}([0,T]) = \big(L^q ( [0,T]; \Cc^{-\eps}(\M)\big)^{m \times (m-1)}
\label{XXX}
\end{align}
and set
\begin{align}  \| \Xi \|_{\mathcal Z_T^{\eps,q}} := \underset{\substack{0\le j_1\le m\\0\le j_2 \le m-1}}{\max} \ \| z_{j_1,j_2} \|_{L^q_T \Cc^{-\eps}}.
\label{semi}
\end{align}
Similarly, for $T = \infty$, we define the space $\mathcal Z^{\eps,q}(\R_+)$ as the topological vector space 
\begin{align*}
\mathcal Z^{\eps,q}(\R_+) =  \big(L_{\text{loc}}^q (\R_+; \Cc^{-\eps}(\M)\big)^{m \times (m-1)}
\end{align*}
whose topology is induced by the family of seminorms \eqref{semi}.

We then have the following well-posedness result for the Cauchy problem for \eqref{model equation}.
\begin{proposition} \label{PROP:LWP1}Let $\alpha_1, \beta_1>0$ and $\alpha_2,\beta_2\in\R$. Let $m \geq 2$ and $s_0 < 0$ be such that 
\begin{equation*}
-\frac{2}{2m-1}<s_0<0
\end{equation*}
and let $0 < \eps \ll 1 \ll q$. Then, the equation \eqref{model equation} is locally well-posed in $\Cc^{s_0}(\M) \times \mathcal Z^{\eps,q}([0,1])$. More precisely, 
given an enhanced data set: 
\begin{align*}
 (v_0, \Xi) \in \Cc^{s_0}(\M) \times \mathcal{Z}^{\eps,q}([0,1]).
\end{align*}
with $\Xi = \{\Xi_{j_1,j_2}\}_{0\le j_1\le m,0\le j_2 \le m-1}$ \textup{(}with $\Xi_{m,m-1} = 1$\textup{)}, there exist a time $T = T(\|v_0\|_{\Cc^{s_0}} , \|\Xi\|_{\mathcal{Z}^{\eps,q}}) \in (0, 1]$\footnote{Here, we use the shorthand notation $\mathcal Z^{\eps,q}$ for $\mathcal Z^{\eps,q}([0,1])$}
and a unique solution $v$ to \eqref{model equation} in the class
\begin{align}
C(  [0,T]; \Cc^{s_0}(\M)) \cap  C(  (0,T]; \Cc^{2\e }(\M)).
\label{Z1}
\end{align}
In particular, the uniqueness of $v$ 
holds in the entire class \eqref{Z1}.
Furthermore, the solution map 

\noi
\begin{align}
(v_0,\Xi) \in \Cc^{s_0}(\M) \times \mathcal{Z}^{\eps,q}([0,1])
\mapsto 
v \in C(  [0,T]; \Cc^{s_0}(\M)) \cap  C(  (0,T]; \Cc^{2\e }(\M))
\label{solmap1}
\end{align}

\noi
is locally Lipschitz continuous.
%
\end{proposition}

\medskip
Here, for $ T >0$, a solution to \eqref{model equation} on $[0,T]$ is defined to be a function $v$ satisfying the Duhamel formulation
\begin{equation}\label{Duhamel}
v(t) = S(t) v_0 + \int_{0}^t S(t-t') F(v, \Xi) dt' \quad \text{for all $t \in [0,T]$},
\end{equation}
where $F$ is as in \eqref{def F} and $S(t)$ is as in \eqref{def semi-group}.

We will build a solution to \eqref{model equation} in the space
\begin{align*}
 X^{s_0, 2\e}_T=C \left(  [0,T]; \Cc^{s_0}(\M)   \right) \cap  C (  (0,T]; \Cc^{2\e }(\M)),
 \end{align*}
  endowed with the norm
\begin{align}\label{XT}
 \| v \|_{X^{s_0, 2\e}_T} = \| v \|_{L^{\infty}_T \Cc^{s_0}} +  \underset{0 < t \le T}{\sup} t^{\frac{2\e-s_0}2} \| v \|_{\Cc^{2\e}}.  
 \end{align}
 
We will use the following elementary lemma.
\begin{lemma}\label{LEM:tech lem 1}
Let $0< \theta_1, \theta_2 <1$. Then the following estimate holds for any $t>0$:
\[  \int_{0}^t (t-s)^{- \theta_1} s^{- \theta_2} ds \les t^{1 - \theta_1 - \theta_2}. \]
\end{lemma}

\medskip
\begin{proof}[Proof of Proposition \ref{PROP:LWP1}]
Let the parameters $\alpha_1, \beta_1, \alpha_2,\beta_2, m, p, q, \eps$ and the data $(v_0, \Xi)$ be as in the statement.

First, we show the existence of the solution in the Banach space $X_T^{s_0, 2\e}$ for some $T \in (0,1]$ to be determined later. To do so, we show that the following map acts as a contraction on a ball of $X_T^{s_0, 2\e}$:
\begin{align}\label{Gamma}
\G v(t) = S(t) v_0 + \int_{0}^t S(t-t') F(v, \Xi) dt', \quad \text{for all $t \in [0,T]$}. 
\end{align}

Let $R = \max(\|v_0\|_{C^{s_0}_x}, \|\Xi\|_{\mathcal Z^{\eps,q}})$ and fix $R_0 > R\ge 1$ to be chosen later.

We first estimate the linear term in the right-hand side of \eqref{Duhamel} using Schauder's estimate (Lemma~\ref{LEM:semigroup}~(i)) for $0 < t \leq T$:
\begin{align}\label{LWP1 1}
 \big\| S(t) v_0 \big\|_{\Cc^{2\e}} \les t^{\frac{s_0-2\e}{2}}  \| v_0 \|_{\Cc^{s_0}}.
\end{align}
As for the nonlinear term, we denote by $B_{R_0}$ the ball of radius $R_0$ in $X_T^{s_0, 2\e}$, and for $v \in B_{R_0}$,  using the linear and bilinear estimates of Lemma~\ref{LEM:semigroup}~(i) and Lemma~\ref{LEM:Besov}~(i)-(ii) with the definitions \eqref{def F} of $F$ and \eqref{XT} of the norm in $X^{s_0,2\e}_T$, we have
\begin{align}\label{LWP1 2}
\begin{split}
\big\| F(v,\Xi) (t') \big\|_{\Cc^{-\e}} & \les  \sum_{\substack{ 0 \leq j_1 \leq m \\  0 \leq j_2 \leq m-1}}  \|  v(t') \|^{j_1+j_2}_{\Cc^{2\e}}   \|\Xi_{m-j_1,m-j_2-1}(t') \|_{\Cc^{-\e}}  \\
& \les  R_0^{2 m - 1} (t')^{(2m-1)\frac{s_0-2\e}{2}} \sum_{\substack{ 0 \leq j_1 \leq m \\  0 \leq j_2 \leq m-1}}   \| \Xi_{j_1,j_2} (t') \|_{\Cc^{-\e}}
\end{split}
\end{align}
for $0<t'\le T$. Hence, combining \eqref{LWP1 1}, \eqref{LWP1 2}, Lemma~\ref{LEM:semigroup}~\textup{(i)}, H\"older's inequality in time and Lemma~\ref{LEM:tech lem 1} yields
\begin{align} \label{LWP1 3}
\begin{split}
 &t^{\frac{2 \e - s_0}{2}} \big\| \G v(t) \big\|_{\Cc^{2\e}}  \\
 &\les  \| v_0 \|_{\Cc^{s_0}} +  t^{\frac{2 \e - s_0}{2}}  R_0^{2 m - 1}  \int_0^t (t-t')^{-\frac{3 \e}{2}} (t')^{(2m-1)\frac{s_0-2\e}{2}} \sum_{\substack{ 0 \leq j_1 \leq m \\  0 \leq j_2 \leq m-1}}   \| \Xi_{j_1,j_2} (t') \|_{\Cc^{-\e}} dt'    \\
&  \les   \| v_0 \|_{\Cc^{s_0}} +  t^{\frac{2 \e - s_0}{2}}  R_0^{2 m - 1} \Big( \int_0^t (t-t')^{- q' \frac{3 \e}{2}} (t')^{q' (2m-1)\frac{s_0-2\e}{2}} dt' \Big)^{\frac{1}{q'}}  \|\Xi\|_{\mathcal Z^{\eps,q}_{T_0}}    \\
& \les   \| v_0 \|_{\Cc^{s_0}} +  R_0^{2 m-1}R \  t^{ \frac{1}{q'} - \frac{3  \e}{2}  + (2m-1 )\frac{s_0 - 2 \e}{2}}
\end{split}
\end{align}
for any $v\in B_{R_0}$ and $0 \leq t \leq T$. Here, $q'$ is the dual exponent of $q$. Similarly, for $\e < -s_0$ and $0<t\leq T$, we have that
\begin{align}\label{LWP1 4}
\begin{split}
\big\| \G v(t) \big\|_{\Cc^{s_0}} & \les  \| v_0 \|_{\Cc^{s_0}} +  \int_0^t \big\|   F(v,\Xi)(t')  \big\|_{\Cc^{-\e}} dt' \\
& \les  \| v_0 \|_{\Cc^{s_0}} + R_0^{2m-1}R  t^{ \frac{1}{q'} +  \frac{s_0 - 2 \e}{2}(2m-1)  }. 
\end{split}
\end{align}
Thus, by choosing then $\e$ small enough and $q$ large enough so that the exponent in the powers of $t$ in \eqref{LWP1 3} and \eqref{LWP1 4} are positive, we deduce the existence of a constant $C>0$ such that
\begin{align}\label{LWP1 5}
\big\| \G v \big\|_{X^{s_0, 2\e}_T} \leq C \| v_0 \|_{\Cc^{s_0}} + C  R_0^{2m-1} R T^{\theta },
\end{align}
with $\theta =  \big( 1 + (2m-1)\frac{s_0-2\e}{2} \big) >0 $. We finally choose $R_0 = 2 C R$ and $T\in (0,T_0]$ small enough such that $R^{2m} T^{\theta} \ll \frac{R}{2}$. This leads to
\begin{align}\label{LWP1 5b}
\big\| \G v \big\|_{X^{s_0, 2\e}_T} \le R_0
\end{align}
for any $v\in B_{R_0}$.

In order to conclude that $\G$ maps $B_{R_0}$ to $B_{R_0}$, we need to check the continuity in time of $\G v$. Fix $t \in (0,T) $ and a small $h>0$. Let $0 < \dl \ll_{\e, s_0} 1$. By Lemma~\ref{LEM:semigroup}~\textup{(i)-(iii)} and Lemma \ref{LEM:Besov}~\textup{(i)}, we can estimate similarly as before
\begin{align*}
& t^{\frac{2\e - s_0}{2}} \big\| \G v(t+h) - \G v(t) \big\|_{\Cc^{2\e}}  \\  & \les \big\| ( S(t+ h) - S(t)) v_0 \big\|_{\Cc^{s_0}}  + t^{\frac{2\e - s_0}{2}} \int_{0}^t  \big\|  \big( S(t+h-t') -S(t-t') \big) F(v,\Xi)(t') \big\|_{\Cc^{2\e}} \\
& \les \big\| ( S(h) - \operatorname{Id}) v_0 \big\|_{\Cc^{s_0}} + t^{\frac{2\e - s_0}{2}} |h|^{\dl} \int_{0}^t  \big\|   S(t-t') F(v,\Xi)(t') \big\|_{\Cc^{2\e + 2 \del}} \\
& \les \big\| ( S(h) - \operatorname{Id}) v_0 \big\|_{\Cc^{s_0}} + t^{\frac{2\e - s_0}{2}} |h|^{\dl} \int_{0}^t (t-t')^{\frac{-3\e -2\dl}{2}} \big\|  F(v,\Xi)(t') \big\|_{\Cc^{- \e}} \\
& \les \big\| ( S(h) - \operatorname{Id}) v_0 \big\|_{\Cc^{s_0}} + |h|^{\dl} R_0^{2 m-1} R  T^{ \frac{1}{q'} - \frac{3 (\e - 2 \dl) }{2}  + \frac{s_0 - 2 \e}{2}(2m-1) }.
\end{align*}
Therefore, taking the supremum in $t \in (0,T]$ in the above and using Lemma \ref{LEM:Besov}~\textup{(iii)} shows the right continuity on $(0;T)$. A similar computation shows the left conitnuity on $(0;T]$, therefore $\G v \in C \left(  (0,T]; \Cc^{2\e }(\M)   \right)$. The continuity of the map $t \in [0,\infty) \mapsto \G v(t) \in \Cc^{s_0}(\M)$ is shown by similar arguments. Hence, $\G$ maps $B_{R_0}$ to itself.

Next, we show that $\G$ is a contraction on $B_{R_0}$. Using the algebraic identity
\begin{align*}
a^{j_1} \cj{a}^{j_2} - b^{j_1} \cj{b}^{j_2} = (a-b) \cj{a}^{j_2} \sum_{k=0}^{j_1-1} b^k a^{j_1-1-k} + \cj{(a-b)} b^{j_1} \sum_{k=0}^{j_2-1} \cj{b}^k \cj{a}^{j_2-1-k},
\end{align*}
which holds for any $a,b\in\C$, and then proceeding as above, we have
\begin{align}\label{LWP1 6}
&\big\| F(v_2, \Xi)(t) - F(v_1,\Xi)(t)  \big\|_{\Cc^{-\e}}\notag\\& \qquad\les t^{(2m-1) \frac{s_0 - 2 \e}{2} } R_0^{2m-2} \Big(  \sum_{\substack{ 0 \leq j_1 \leq m \\  0 \leq j_2 \leq m-1}}   \| \Xi_{j_1,j_2} (t) \|_{\Cc^{-\e}} \Big) \| v_2- v_1 \|_{X^{s_0, 2 \e}_T}
\end{align}
for $v_1,v_2 \in B_{R_0}$ and $0 < t \leq T$. By doing computations similar to those that lead to \eqref{LWP1 3} and \eqref{LWP1 4} we get
\begin{align*}
\big\| \G v_2 - \G v_1 \big\|_{X^{s_0, 2\e}_T} \leq  C  R_0^{2m-2} R T^{\theta } \| v_2 - v_1 \|_{X^{s_0, 2\e}_T},
\end{align*}
with $ \theta  >0$ as above. Hence, by choosing $T$ smaller if necessary so that $  C  R_0^{2m-2} R  T^{\theta } \leq \frac12 $, we obtain a unique solution $v\in B_{R_0}$ to the fixed point problem \eqref{Duhamel} by Banach fixed-point theorem. 

The rest of the proof follows from similar and standard arguments and we omit details.
\end{proof}

With the general local well-posedness result of Proposition~\ref{PROP:LWP1}, we can prove our first main result.
\begin{proof}[Proof of Theorem \ref{THM : LWP theory 1}]Let $\alpha_1, \beta_1, \gamma>0$ and $\alpha_2,\beta_2\in\R$. Let $m\ge 2$, $-\frac2{2m-1}<s_0<0$ and $u_0\in \Cc^{s_0}(\M)$. Fix then $0<\eps\ll 1 \ll q$ as in Proposition~\ref{PROP:LWP1}. For any $M \geq 1$ we define a subset $\Si_M\subset\Omega$ as
\begin{align}\label{sol1}
\begin{split}
\Si_M = \Big\{ & \om \in \O,~\Psi \in C\left( [0,1]; \Cc^{-\e}(\M) \right),~\forall N\ge 1 ~\big\|\Psi_N-\Psi\|_{C([0,1])\Cc^{-\eps}}\le M N^{-\frac{\dl}2},  \\
& \forall 0\le j_1 \le m,0\le j_2\le m-1,~ \big\| :\Psi^{j_1} \cj{\Psi}^{j_2}: \big\|_{L^q_{[0,1]} \Cc^{-\e}} \leq M, \\
& \text{ and } \forall N\ge 1,~\big\|  :\PN^{j_1} \cj{\Psi_N}^{j_2}: - : \Psi^{j_1} \cj{\Psi}^{j_2}  : \big\|_{L^{q}_{[0,1]}\Cc^{-\e}}\le M N^{-\frac{\dl}2}\Big\},
\end{split}
\end{align} 
where $\dl>0$ is as in Proposition~\ref{PROP:construction stochastic objects}. Then, by the tail estimates \eqref{deviation}, \eqref{deviation2} and \eqref{deviation3} in Proposition~\ref{PROP:construction stochastic objects}, we have
\[ \Prob \big( \O \setminus{ \Si_M } \big) \leq C e^{-c M^{\frac{2}{m}} }.\]

For $N \in \N \cup \{\infty\}$, let $u_N = \Psi_N + v_N$, where $\Psi_N$ is as in \eqref{truncated stochastic convolution}, $v_N$ solves \eqref{vN} with the initial data $u_0 - \Psi_N(0)$ and $\Xi_N = \{:\!\PN^{j_1} \cj{\Psi_N}^{j_2}\!\!:\}_{0\le j_1\le m,0\le j_2 \le m-1}$. Note that for finite $N$, $u_N$ solves~\eqref{renormalized truncated SCGL 1}.

Let $M \ge 1$. By construction of the set $\Sigma_M$ and the stability results of Proposition \ref{PROP:LWP1}, there exists $ 0 < T= T(M) \leq 1$ such that $v_N \to v_{\infty}$ in $C \left(  [0,T]; \Cc^{s_0}(\M)   \right) $ as $N \to \infty$, on $\Sigma_M$. Since $\Psi_N \to \Psi_{\infty}$ in $C \left(  [0,T]; \Cc^{-\eps}(\M)   \right)$, as $N \to \infty$ on $\Sigma_M$, we have that
\[u_N\to u \text{ in }C([0,T];\Cc^{s_0}(\M)),\]
as $N \to \infty$, on $\Sigma_M$. Here, $u = u_{\infty} = \Psi_{\infty} + v_{\infty}$.

Finally, Borel-Cantelli's lemma shows that
\[ \Si \deff \underset{M \to \infty}{\liminf} \ \Si_M \]
is of full probability, and from the previous discussion, we have pathwise local well-posedness on $\Si$.
\end{proof}
Following the argument in \cite{MW2}, our strategy to prove deterministic global well-posedness is to prove an a priori $L^p$ estimate on the growth in time of a solution to \eqref{model equation}; see the next subsection. Thus, we first need to have a control on the existence time of the solutions depending on the $L^p$-norm of their initial data.
\begin{proposition} \label{PROP:LWP2}
Let $\alpha_1, \beta_1>0$ and $\alpha_2,\beta_2\in\R$. Fix an integer  $m \geq 2$, $p > 2m-1$ and $0 < \eps \ll 1 \ll q$. Then, the equation \eqref{model equation} is locally well-posed in $L^p(\M) \times \mathcal Z^{\eps,q}(\R_+)$.\footnote{Here, we introduce the data $\Xi$ on $\R_+$ since we need to consider $\Xi$ over abitrary large time interval in the blowup alternative below.} More precisely, 
given an enhanced data set: 
\begin{align*}
 (v_0, \Xi) \in L^p(\M) \times \mathcal{Z}^{\eps,q}([0,1]), 
\end{align*}
with $\Xi = \{\Xi_{j_1,j_2}\}_{0\le j_1\le m, 0\le j_2 \le m-1}$ \textup{(}with $\Xi_{m,m-1} = 1$\textup{)} there exist $T = T\big(\|v_0\|_{L^p} , \|\Xi\|_{\mathcal{Z}^{\eps,q}}\big) \in (0, 1]$ and a unique solution $v$ to \eqref{model equation} in the class
\begin{align}
C([0,T]; L^p(\M)) \cap C ((0,T]; \Cc^{2\e }(\M)).
\label{Z2}
\end{align}
In particular, the uniqueness of $v$ 
holds in the entire class \eqref{Z2}.
Furthermore, the solution map 
\begin{align}
(v_0,\Xi) \in L^p(\M) \times \mathcal{Z}^{\eps,q}([0,1])
\mapsto 
v \in C([0,T]; L^p(\M)) \cap C ((0,T]; \Cc^{2\e }(\M))
\label{solmap2}
\end{align}

\noi
is locally Lipschitz continuous.

Lastly, there is the following blow-up alternative: for any fixed $T_0 >0$, let the time $\T = \T \big(T_0,\| v_0 \|_{L^p},\|\Xi\|_{\mathcal Z^{\eps,q}}\big) \in (0,T_0]$ denotes the maximal time of existence of $v$ on $[0,T_0]$. Then, we have
\begin{align*}
\T=T_0 \qquad\text{or}\qquad\sup_{0\le t<\T}\|v(t)\|_{L^p}=+\infty.
\end{align*}
\end{proposition}

\medskip

We define for $T>0$, $\eps>0$ and $1 \le p < \infty$ the Banach space 
$$Y^{\eps}_p([0,T])=C([0;T];L^p(\M))\cap C((0,T];\Cc^{2\eps}(\M)),$$ endowed with the norm
\begin{align*}
\|v\|_{Y^{\eps}_p(T)} =  \| v \|_{L^{\infty}_T L^p_x } + \sup_{0<t\le T}t^{\eps+\frac1p}\|v(t)\|_{\Cc^{2\eps}}.
\end{align*} 

\begin{proof} 
Let the parameters $\alpha_1, \beta_1, \alpha_2,\beta_2, m, p, q, \eps$ and the data $(v_0, \Xi)$ be as in the statement.

We first construct solutions in $Y^{\eps}_p([0,T])$ for some small $0< T \ll 1$. We only sketch the argument here (i.e. we prove a priori estimates) as the rest follows as the proof Proposition \ref{PROP:LWP1}. Using Lemma~\ref{LEM:Besov}~\textup{(iii)} together with Lemma~\ref{LEM:semigroup}~\textup{(i)} and the uniform boundedness of the smooth frequency projectors $\{\P_N\}_{N \in \N}$, we have
\begin{align}\label{LWP2-1}
\|S(t) v_0 \|_{\Cc^{2\e}} \les \|S(t) v_0 \|_{\Bb^{2\e + \frac2p}_{p, \infty}} \les t^{- \e - \frac1p} \|v_0\|_{L^p}.
\end{align}
Hence, letting $\G$ be as in \eqref{Gamma} above, we can proceed as for \eqref{LWP1 3} by using \eqref{LWP2-1} and Young's inequality to estimate
\begin{align}
\begin{split}
t^{\e + \frac1p} \| \Gamma v(t) \|_{\Cc^{2\e}} & \les  \|v_0\|_{L^p} +  \|v\|_{Y^{\eps}_p(T)} ^{2m-1} \cdot t^{\e + \frac1p} \int_{0}^t (t-t')^{-\frac{3 \e}{2}} (t')^{-(2m-1)(\e + \frac{1}{p})} \| \Xi (t')\|_{\Cc^{-\e}} dt' \\
& \les \|v_0\|_{L^p} + \|v\|_{Y^{\eps}_p(T)} ^{2m-1} \, \|\Xi\|_{\mathcal Z^{\eps,q}} T^{ \frac{1}{q'} -\frac{3\e}{2} - (2m-2) (\e + \frac1p)}
\end{split}
\label{LWP2-2}
\end{align}
 for $0 < t \le T$. Note that we used Lemma~\ref{LEM:tech lem 1} in the second estimate, which was allowed under the assumption $ \frac{2m-1}{p} < 1 $ provided that $\eps$ is small enough. 

Similarly, we also get
\begin{align}
\|\Gamma v\|_{L^\infty_TL^p_x} \les \|v_0\|_{L^p} + \|v\|_{Y^{\eps}_p(T)} ^{2m-1} \, \|\Xi\|_{\mathcal Z^{\eps,q}} T^{\ta}
\label{LWP2-3}
\end{align}
for all $0\le t \le T$ and a constant $\ta>0$. Hence taking the supremum in $0 < t \le T$ in \eqref{LWP2-2} and $0 \le t \le T$ in \eqref{LWP2-3} and arguing as in Proposition~\ref{PROP:LWP1} shows that $\Gamma$ maps small balls in $Y^{\eps}_q([0,T])$ to themselves. The contraction property may be proved in a similar fashion. This yields solution $v\in Y^{\eps}_p([0,T])$ to \eqref{Duhamel} for \eqref{model equation}. The rest of the proof follows from similar and standard arguments and we omit details.
\end{proof}
\subsection{Global well-posedness}\label{SUBSEC : gwp}
We now discuss the global well-posedness result stated in Theorem~\ref{THM : deterministic GWP}. In this part, we assume $\alpha_1, \beta_1 >0 $, $\al_2 \ge 0$ and denote\footnote{Recall that we set $r=\infty$ when $\alpha_2=0$.} by $r$ the \textit{dissipation-to-dispersion} ratio:
\begin{equation}\label{def of r}
r = \frac{\alpha_1}{\alpha_2} \in (0,\infty].
\end{equation}
Note that the assumption made above on $\alpha_2$ is not actually necessary (i.e. it suffices to assume that $\al_2 \in \R)$. However, by conjugating \eqref{renormalized truncated SCGL 1} and using the fact that the norms that we use are invariant under complex conjugation, we can always assume $\alpha_2 \ge 0$.

Next, recall we constructed a local solution $u = \Psi + v$ to \eqref{WSCGL} as the limit of $u_N=\PN + v_N$, where $v_N$ solves \eqref{vN}. Moreover, for any $t_0>0$ it holds $v_N(t_0)\in L^p(\M)$. Therefore, in view of the blow-up alternative in Proposition~\ref{PROP:LWP2}, the goal is then to obtain uniform (in $N\ge 1$) bounds on $\|v_N(t)\|_{L^p}$ in order to obtain that $v_N$ is well-defined and approximates $v$ up to some arbitrary time $T>0$. Thus, our main result in this subsection is the following global well-posedness for the model equation \eqref{model equation}.
\begin{proposition}\label{PROP:deterministic} 
Let $\al_1, \be_1>0$, $\al_2 \ge 0$ and $\be_2 \in \R$. Fix an integer $m \ge 2$ and $ - \frac{2}{2m-1}<s_0<0$. Assume that the dissipation-to-dispersion ratio $r$ in \eqref{def of r} satisfies \eqref{condition on r}, and take $p\ge 1$ satisfying \eqref{condition-p}.
Then, there exists $0<\eps = \eps(m,p)\ll1$, finite $q = q(p,m)\ge 1$ such that the following holds. Given an enhanced data set:
\begin{align*}
 (v_0, \Xi) \in  \Cc^{s_0}(\M) \times \mathcal Z^{\eps, q}(\R_+), 
\end{align*}
where $\Xi = \{\Xi_{j_1,j_2}\}_{0\le j_1\le m,0\le j_2 \le m-1}$ \textup{(}with $\Xi_{m,m-1} = 1$\textup{)}, there exists a unique solution $v$ to \eqref{model equation} in the class\footnote{Here, spaces of the form $C(\R_+; X)$ (or $C(\R_+^*; X)$), for a Banach space $X$, are endowed with the compact-open topology.}
\begin{align}
C(\R_+; \Cc^{s_0}(\M)) \cap  C(\R_+^{*}; \Cc^{2\e }(\M) ).
\label{Z3}
\end{align}
Moreover, for each $T>0$, there exists a positive constant $B = B(T, \|v_0\|_{\Cc^{s_0}}, \|\Xi\|_{\mathcal Z^{\eps,q_0}_T})$ such that the following a priori estimate holds:
\begin{align}
\sup_{t \in (0,T]}  \|v(t) \|_{\Cc^{2\eps}} \leq B.
\label{boundY}
\end{align}
In particular, the uniqueness of $v$ 
holds in the entire class \eqref{Z3}.
Furthermore, the solution map 

\noi
\begin{align}
(v_0,\Xi) \in \Cc^{s_0}(\M) \times \mathcal{Z}^{\eps,q}([0,T])
\mapsto 
v \in C([0,T]; \Cc^{s_0}(\M)) \cap  C([0,T]; \Cc^{2\e }(\M))
\end{align}
is locally Lipschitz continuous, for every $T>0$.
\end{proposition}

The proof of Proposition~\ref{PROP:deterministic} is a straightforward adaptation of the argument given in \cite{MW2,TsW,Trenberth}. We just need to deal with some extra perturbative terms arising from the magnetic part. As in the aforementioned references, we will split the proof of Proposition~\ref{PROP:deterministic} into several technical lemmas.

\begin{lemma} \label{LEM:a priori estimate}
Let $\al_1, \be_1 >0$, $\al_2 \ge 0$ and $\be_2 \in \R$. Fix an integer $m \geq 2$ and $p \geq 1$ satisfying \begin{align}\label{cond a priori}
 2 -  2 \big( r^2 + 2r \sqrt{1 + r^2} \big)  < p < 2 + 2 \big( r^2 + r \sqrt{1 + r^2} \big).
\end{align}
Let $T>0$, $t_0\in (0,T)$, $\eps >0$ and $v$ be a solution to \eqref{model equation} on $[0,T]$, with enhanced data set $(v_0, \Xi)$, belonging to the class $C([0, T]; \Cc^{2\eps}(\M)) \cap C^1_t C^2_x((t_0, T) \times \M)$.\footnote{Here, we endow spaces of the form $C^1_t C^2_x(I \times \M)$, for some open time interval $I$, with the standard H\"older topology.}\footnote{This regularity condition is purely technical and ensures that $v$ is a classical solution on $(t_0, T)$ so that the computations in the following proofs are valid.} Then, there exists $0 < \eta \ll 1$ such that the following \textit{a priori} bound holds:
\begin{align}\label{aprioriestimate}
\begin{split}
&\frac{1}{p} \Big( \big\| v(t) \big\|_{L^p}^p -  \big\| v(t_0) \big\|_{L^p} ^p \Big)  + \beta_1\int_{t_0}^t \int_\M |v(t')|^{p-2 + 2m}dVdt'\\
&\qquad\qquad +\alpha_1\int_{t_0}^t\int_\M|v(t')|^p|A|^2dVdt' + 4 \eta \alpha_1 \int_{t_0}^t \int_\M  |v|^{p -2}(t')| \nb v|^2(t')dV dt'  \\
&\leq  \int_{t_0}^t \big| \langle F_0(v, \Xi), |v|^{p-2}v \rangle_{L^2(\M)} \big| (t') dt' +|\alpha_1+i\alpha_2|\int_{t_0}^t \big|\langle |v|^{p-2}v A,dv\rangle_{L^2(T^*\M)}\big|(t')dt'\\
&\qquad\qquad\qquad\qquad+|\alpha_1+i\alpha_2|\int_{t_0}^t \big|\langle d^*(vA),|v|^{p-2}v\rangle_{L^2(\M)}\big|(t')dt'
\end{split}
\end{align}
for all $t\in [t_0,T]$, where 
\begin{align}\label{F0}
F_0(v,\Xi) := (\beta_1 + i \beta_2) \sum_{\substack{ 0 \leq j_1 \leq m \\  0 \leq j_2 \leq m-1 \\ (j_1,j_2) \neq (m,m-1) }} \binom{m}{j_1} \binom{m-1}{j_2} v^{j_1} \cj{v}^{j_2}  \Xi_{m-j_1,m-j_2-1}. 
\end{align}
\end{lemma}
\begin{proof}
 We adapt the proof of \cite[Proposition 18]{MW2} as in \cite{TsW,Trenberth}. In view of the regularity of $v \in C^1_t C^2_x((t_0, T) \times \M)$, we can compute
\begin{align*}
&\frac{1}{p} \frac{d}{dt} \big\| v(t) \big\|_{L^p}^p\\
 &  = \Re\langle \dt v,|v|^{p-2} v\rangle_{L^2(\M)}\\
& = \Re\langle (\alpha_1 + i\alpha_2)\Dla v -F(v, \Xi),|v|^{p-2}v\rangle_{L^2(\M)}
\end{align*}
for each $t_0 < t < T$. Now by the definition \eqref{Delta} of the magnetic Laplace-Beltrami operator, we can continue with
\begin{align}
\begin{split}
& \frac{1}{p} \frac{d}{dt} \big\| v(t) \big\|_{L^p}^p\\
&= - \Re\langle (\alpha_1+i\alpha_2) D_Av,D_A\big[|v|^{p-2}v\big]\rangle_{L^2(T^*\M)}+\Re\langle F(v, \Xi),|v|^{p-2}v\rangle_{L^2(\M)}\\
&=- \Re\langle (\alpha_1+i\alpha_2) \nabla v,\nabla\big[|v|^{p-2}v\big]\rangle_{L^2(T\M)}- \Re\langle (\alpha_1+i\alpha_2) vA,|v|^{p-2}v A\rangle_{L^2(T^*\M)}\\
&\qquad - \Re \langle (i\alpha_1-\alpha_2) d v,|v|^{p-2}v A\rangle_{L^2(T^*\M)}-\Re \langle (i\alpha_1-\alpha_2) d^*(vA),|v|^{p-2}v\rangle_{L^2(\M)}\\
&\qquad -\Re\langle F(v, \Xi),|v|^{p-2}v\rangle_{L^2(\M)}.
\end{split}
\label{apriori1}
\end{align}
The second term in the right-hand side of \eqref{apriori1} gives a contribution in the left-hand side of \eqref{aprioriestimate}, while the third and fourth terms are found in the right-hand side of \eqref{aprioriestimate}.

As for the first term, we use the Leibniz rule for vector fields and the formula $\Re(z) = \frac12 (z + \cj z)$, for $z \in \C$, as follows.
\begin{align}
&\Re \langle (\alpha_1+i\alpha_2) \nabla v,\nabla\big[|v|^{p-2}v\big]\rangle_{L^2(T\M)} = \Re \langle (\alpha_1+i\alpha_2) \nabla v,\nabla\big[(v \cj v)^{\frac{p-2}{2}}v\big]\rangle_{L^2(T\M)}\notag\\
&\qquad = \frac{p}{2} \alpha_1 \langle |v|^{p-2}  \nb v,\nb v\rangle_{L^2(T\M)}\notag\\
&\qquad\qquad + \frac{p-2}{4} \alpha_1 \Big(\langle |v|^{p-4} \nb v,v^2\nb \vb\rangle_{L^2(T\M)}  + \langle |v|^{p-4} v^2  \nb\vb,\nb v\rangle_{L^2(T\M)}\Big)  \label{apriori2} \\
& \qquad\qquad \qquad   +i \frac{p-2}{4} \alpha_2   \Big(\langle |v|^{p-4}  \nb v,v^2\nb \vb\rangle_{L^2(T\M)} - \langle |v|^{p-4}  v^2 \nb \vb,\nb v\rangle_{L^2(T\M)}\Big).       \notag
\end{align}
To further analyse the terms above, we will use the following algebraic identities, which follow from straightforward computations as in the Euclidean case, using the expression $\langle X,Y\rangle = \gm_{j,k}X^j\cj{Y}^k$ in local coordinates for the inner product on vector fields:
\begin{align}
\begin{split}
 & v^2 \langle \nb \vb,\nb v\rangle_{T\M} + \vb^2 \langle \nb v,\nb\vb\rangle_{T\M} \\
 & \qquad \qquad =  \langle v \nb \vb - \vb \nb v,\vb \nb v - v \nb \vb\rangle_{T\M} + 2 |v|^2 \langle\nb v ,\nb v\rangle_{T\M},
\end{split}
 \label{Yeq1}
 \end{align}
 \begin{align}
 \vb ^2 \langle\nb v, \nb \vb\rangle_{T\M} - v^2 \langle \nb \vb,\nb v\rangle_{T\M}= \langle\nb( |v|^2),  v \nb \vb - \vb \nb v \rangle_{T\M},
 \label{Yeq2}
 \end{align}
 and
 \begin{align}
4 |v|^2 \langle\nb v,\nb v\rangle_{T\M} =  \langle \nb (|v|^2),\nb (|v|^2)\rangle_{T\M} - \langle v \nb \vb - \vb \nb v,\vb \nb v - v \nb \vb\rangle_{T\M}.
\label{Yeq3}
\end{align}

By using \eqref{Yeq1} and \eqref{Yeq2}, we can thus write \eqref{apriori2} as
\begin{align}
\begin{split}
& \Re \langle (\alpha_1+i\alpha_2) \nabla v,\nabla\big[|v|^{p-2}v\big]\rangle_{L^2(T\M)}   \\
&  =  (p-1)\alpha_1 \langle |v|^{p-2}  \nb v ,\nb v\rangle_{L^2(T\M)} + \frac{p-2}{4} \alpha_2 i \langle |v|^{p-4}  \nb (|v|^2) ,v \nb \vb - \vb \nb v\rangle_{L^2(T\M)} \\
&  \qquad \quad + \frac{p-2}{4} \al_1 \langle|v|^{p-4}  (v \nb \vb - \vb \nb v)  ,\vb\nb v - v \nb \vb \rangle_{L^2(T\M)}.
\end{split}
\label{Yeq4}
\end{align}
Fix $\eta >0$. By \eqref{Yeq3}, we have
\begin{align}
\begin{split}
& (p-1)\alpha_1 \langle|v|^{p-2}  \nb v ,\nb v\rangle_{L^2(T\M)} \\
& = 4 \eta \al_1 \langle|v|^{p-2} nb v ,\nb v\rangle_{L^2(T\M)} + \Big( \frac{p-1}{4} - \eta \Big) \al_1 \langle |v|^{p-4} \big( 4 |v|^2  \nb v\big) ,\nb v\rangle_{L^2(T\M)} \\
& = 4 \eta \al_1 \langle |v|^{p-2} \nb v ,\nb v\rangle_{L^2(T\M)}+ \Big(\frac{p-1}{4} - \eta \Big) \alpha_1\langle|v|^{p-4}  \nb (|v|^2),\nb (|v|^2)\rangle_{L^2(T\M)} \\
& \qquad \quad - \Big(\frac{p-1}{4} - \eta \Big)\alpha_1 \langle |v|^{p-4} ( v \nb \vb - \vb \nb v) ,\vb \nb v - v \nb \vb \rangle_{L^2(T\M)}.
\end{split}
\label{Yeq5}
\end{align}
Thus, combining \eqref{Yeq4} and \eqref{Yeq5} yields
\begin{align}
\begin{split}
& \Re \langle (\alpha_1+i\alpha_2) \nabla v,\nabla\big[|v|^{p-2}v\big]\rangle_{L^2(T\M)} \\
& \qquad \qquad =  4 \eta \alpha_1 \langle|v|^{p-2}\nb v,\nb v\rangle_{L^2(T\M)} + \int_{\M} |v|^{p-4} B_{p, \eta}(X,Y) dV,
\end{split}
\label{apriori3}
\end{align}
where we define the vector fields
\[ X = i \big( \vb \nb v - v \nb \vb \big) \quad , \quad Y = \nb (|v|^2),\]
and $B_{p, \eta}(X,Y)$ is the quadratic form on the space of vector fields defined as
\[  B_{p, \eta}(X,Y) = \Big( \frac{1}{4} - \eta \Big)\alpha_1 \langle X,X\rangle_{T\M)} + \frac{p-2}{4} \alpha_2 \langle X,Y\rangle_{T\M} + \Big( \frac{p-1}{4} - \eta \Big) \alpha_1 \langle Y,Y\rangle_{T\M}. \]
First, we note that since $X$ and $Y$ are real-valued, so is $B_{p, \eta}(X,Y)$. If $\alpha_2=0$, then the form $B_{p, \eta}(X,Y)$ takes non-negative values for $\eta$ small enough. Otherwise, $\al_2 >0$ and the quadratic form takes non-negative values if and only if the matrix
\[  \begin{pmatrix}
 \Big( \frac{1}{4} - \eta \Big)r  & \frac{p-2}{8} \\ \frac{p-2}{8}& \ \ \Big( \frac{p-1}{4} - \eta \Big)\end{pmatrix} \]
is non-negative definite, which is the case under \eqref{cond a priori} for $\eta$ small enough. Going back to \eqref{apriori3}, this yields
\begin{align}
- \Re(\alpha_1+i\alpha_2)\langle \nb v,\nb\big[|v|^{p-2}v\big]\rangle_{L^2(T\M)}\le -4\eta\alpha_1\int_\M|v|^{p-2}|\nb v|^2dV
\label{Yeq6}
\end{align}
for some $\eta>0$ small enough. 

Finally, for the last term on the right-hand side of \eqref{apriori1}, we simply note by \eqref{def F} and \eqref{F0}, we have that
\begin{align}
F(v,\Xi)=(\beta_1+i\beta_2)|v|^{2m-2}v+F_0(v,\Xi).
\label{Yeq7}
\end{align}
Hence, \eqref{aprioriestimate} follows from \eqref{apriori1}, \eqref{Yeq6}, \eqref{Yeq7} and integrating in time.
\end{proof}
The next result provides a control for the growth in time of the $L^p$ norm. 
\begin{lemma}\label{LEM:a priori 2} 
Let $\al_1, \be_1>0$, $\al_2 \ge 0$ and $\be_2 \in \R$. Fix an integer $m \ge 2$ and $p\ge 1$ satisfying \eqref{cond a priori}. Let $T>0$, $t_0 \in (0,T]$, $\eps >0$ and $v$ be a solution to \eqref{model equation} on $[t_0,T]$, with enhanced data set $(v_0,\Xi) \in \D'(\M) \times \mathcal Z^{\eps, q}([0,T])$, belonging to the class $C([t_0, T]; \Cc^{2\eps}(\M)) \cap C^1_t C^2_x((t_0, T) \times \M)$. Then, there exist paremeters $0<\eps = \eps(m,p)\ll1$ and $1 \le q = q(p,m) < \infty$ depending only on $(p,m)$ and a constant $C = C\big(m,p,A)>0$ depending only on $(p,m,A)$ such that we have
\[ \| v(t)  \|_{L^p} \leq \| v(t_0)  \|_{L^p} + C \cdot \big((t-t_0)+\|\Xi\|_{\mathcal Z_T^{\eps, q}}^q\big)\] 
for all $t \in [t_0,T]$
\end{lemma}
\begin{proof}
Let the parameters $\al_1, \al_2, \be_1, \be_2, m, T, t_0, p, q, \eps$ and the data $(v_0, \Xi)$ be as in the statement. Our strategy is standard and amounts to showing that the terms on the right-hand-side of \eqref{aprioriestimate} may estimated by a small fraction of the (positive) left-hand-side and another term independent of $v$.

We begin by estimating the first term on the right-hand-side of \eqref{aprioriestimate}. By the definition of $F_0$ \eqref{F0}, we have
\begin{align}
\begin{split}
&\big| \int_{t_0}^t\langle F_0(v, \Xi)(t'), |v|^{p-2}v(t') \rangle_{L^2(\M)}dt' \big|\\
& \les \sum_{\substack{ 0 \leq j_1 \leq m \\  0 \leq j_2 \leq m-1 \\ (j_1,j_2) \neq (m,m-1) }} \int_{t_0}^t\big|\langle v^{j_1}(t') \cj{v}^{j_2}(t')  \Xi_{m-j_1,m-j_2-1}(t'),|v|^{p-2}v(t')\rangle_{L^2(\M)}\big|dt'.
\end{split}
\label{bddY0}
\end{align}

Fix $0\le j_1\le m,0\le j_2\le m-1$ with $j_1+j_2\le 2m-2$ and $0<\eps\ll1$ to be chosen later. In what follows, we estimate the integrand on the right-hand-side of \eqref{bddY0} while keeping its dependence in $t_0 \le t' \le t$ implicit to ease our notations. By duality, the product rule of Lemma~\ref{LEM:Besov}~(iv), the interpolation estimate of Lemma~\ref{LEM:Besov}~(v), and the Leibniz rule, we have
\begin{align}
\begin{split}
&\big|\langle v^{j_1}\cj{v}^{j_2}  \Xi_{m-j_1,m-j_2-1},|v|^{p-2}v\rangle_{L^2}\big|\\
&\qquad\le \|\Xi_{m-j_1,m-j_2-1}\|_{\Cc^{-\eps}}\big\|v^{\frac{p}2+j_1}\cj{v}^{\frac{p}2-1+j_2}\big\|_{B^{\eps}_{1,1}}\\
&\qquad\les \|\Xi_{m-j_1,m-j_2-1}\|_{\Cc^{-\eps}}\bigg\{\Big(\int_\M|v|^{p-1+j_1+j_2}dV\Big)^{(1-\eps)}\\
&\qquad\qquad\qquad\qquad \times \Big(\int_\M |v|^{p-2+j_1+j_2}|\nabla v|(t)dV\Big)^{\eps} + \int_\M|v|^{p-1+j_1+j_2}dV\bigg\}.
\end{split}
\label{bddY1}
\end{align}
By H\"older's inequality, we have that
\begin{align}
\bigg(\int_\M|v|^{p-1+j_1+j_2}dV\bigg)\les \bigg(\int_\M|v|^{p-2+2m}dV\bigg)^{\frac{p-1+j_1+j_2}{p-2+2m}}
\label{bddY2}
\end{align}
and by Cauchy-Schwarz inequality,
\begin{align}
\begin{split}
& \bigg(\int_\M |v|^{p-2+j_1+j_2}|\nabla v|dV\Big)^{\eps} \\
& \qquad \qquad \qquad \les \bigg(\int_\M|v|^{p-2}|\nabla v|^2dV\bigg)^{\frac\eps2}\Big(\int_\M |v|^{p-2+2(j_1+j_2)}dV\bigg)^{\frac\eps2}.
\end{split}
\label{bddY2b}
\end{align}
We estimate the last integral above, by Sobolev's inequality as follows:
\begin{align}
\begin{split}
&\bigg(\int_\M |v|^{p-2+2(j_1+j_2)}dV\bigg)^{\frac\eps2}\\
&\qquad=\|v^{\frac{p}2}\|_{L^{\frac2p(p-2+2(j_1+j_2))}}^{\frac\eps{p}(p-2+2(j_1+j_2))} \les \|v^\frac{p}2\|_{H^1}^{\frac\eps{p}(p-2+2(j_1+j_2))}\\
&\qquad\les \bigg(\int_\M|v|^pdV\Big)^{\frac\eps{2p}(p-2+2(j_1+j_2))}+\Big(\int_\M |v|^{p-2}|\nabla v|^2dV\bigg)^{\frac\eps{2p}(p-2+2(j_1+j_2))}.
\end{split}
\label{bddY3}
\end{align}
Hence, combining \eqref{bddY1}, \eqref{bddY2}, \eqref{bddY2b} and \eqref{bddY3} together yields
\begin{align}
\begin{split}
&\big|\langle v^{j_1} \cj{v}^{j_2} \Xi_{m-j_1,m-j_2-1},|v|^{p-2}v\rangle_{L^2}\big|\\
&\les \|\Xi_{m-j_1,m-j_2-1}\|_{\Cc^{-\eps}}\bigg\{\bigg(\int_\M|v|^{p-2}|\nabla v|^2dV\bigg)^{\frac{p-1+j_1+j_2}{p-2+2m}}\\
&\hspace{3.2mm}+\bigg(\int_\M|v|^{p-2+2m}dV\bigg)^{(1-\eps)\frac{p-1+j_1+j_2}{p-2+2m}+\frac\eps2\frac{p-2+2(j_1+j_2)}{p-2+2m}}\bigg(\int_\M|v|^{p-2}|\nabla v|^2dV\bigg)^{\frac\eps2}\\
&\hspace{3.2mm} +\bigg(\int_\M|v(t)|^{p-2+2m}dV\bigg)^{(1-\eps)\frac{p-1+j_1+j_2}{p-2+2m}}\bigg(\int_\M|v|^{p-2}|\nabla v|^2dV\bigg)^{\frac\eps2+\frac\eps{2p}(p-2+2(j_1+j_2))}\bigg\}.
\end{split}
\label{bddY4}
\end{align}

Now, by using Young's inequality twice and since $j_1 + j_2 \le 2m-2$, we can find $0<\eps\ll1$ small enough such that we have
\begin{align}
\begin{split}
&\|\Xi_{m-j_1,m-j_2-1}(t)\|_{\Cc^{-\eps}}\bigg(\int_\M|v(t)|^{p-2+2m}dV\bigg)^{(1-\eps)\frac{p-1+j_1+j_2}{p-2+2m}+\frac\eps2\frac{p-2+2(j_1+j_2)}{p-2+2m}}\\
&\qquad\qquad\qquad\qquad\qquad \times\bigg(\int_\M|v|^{p-2}|\nabla v|^2(t)dV\bigg)^{\frac\eps2}\\
&\le \delta_1\int_\M|v(t)|^{p-2+2m}dV+C(\delta_1)\|\Xi_{m-j_1,m-j_2-1}\|_{\Cc^{-\eps}}^{q'_1}\bigg(\int_\M|v|^{p-2}|\nabla v|^2dV\bigg)^{\frac\eps2q_1'}\\
&\le \delta_1\int_\M|v|^{p-2+2m}dV+\delta_2\int_\M|v|^{p-2}|\nabla v|^2(t)dV+ C(\delta_1,\delta_2)\|\Xi_{m-j_1,m-j_2-1}\|_{\Cc^{-\eps}}^{q_1'r_1'}
\end{split}
\label{bddY5}
\end{align}
for arbitrary $\delta_1,\delta_2>0$ and where 
\[q_1=\frac{p-2+2m}{(1-\eps)(p-1+j_1+j_2)+\frac\eps2(p-2+2(j_1+j_2))}\] 
and $r_1=\frac2{\eps q'}$. In the above, $C(\delta_1,\delta_2)>0$ is a large constant depending on $p$, $m$, and $\eps$. 

Similarly, provided $\eps$ is small enough, we have
\begin{align}
\begin{split}
&\|\Xi_{m-j_1,m-j_2-1}\|_{\Cc^{-\eps}}\bigg(\int_\M|v|^{p-2+2m}dV\bigg)^{(1-\eps)\frac{p-1+j_1+j_2}{p-2+2m}}\\
&\qquad\qquad\times\bigg(\int_\M|v|^{p-2}|\nabla v|^2dV\bigg)^{\frac\eps2+\frac\eps{2p}(p-2+2(j_1+j_2))}\\
&\le \delta_1\int_\M|v|^{p-2+2m}dV+\delta_2\int_\M|v|^{p-2}|\nabla v|^2dV+C(\delta_1,\delta_2)\|\Xi_{m-j_1,m-j_2-1}\|_{\Cc^{-\eps}}^{q_2'r_2'},
\end{split}
\label{bddY6}
\end{align}
with $q_2=\frac{p-2+2m}{(1-\eps)(p-1+j_1+j_2)}$ and $\frac1{r_2}=q_2'(\frac\eps2+\frac\eps{2p}(p-2+2(j_1+j_2)))$. At last, we also have
\begin{align}
\begin{split}
&\|\Xi_{m-j_1,m-j_2-1}\|_{\Cc^{-\eps}}\bigg(\int_\M|v|^{p-2}|\nabla v|^2dV\bigg)^{\frac{p-1+j_1+j_2}{p-2+2m}}\\
&\le \delta_2\int_\M|v|^{p-2}|\nabla v|^2dV+C(\delta_2)\|\Xi_{m-j_1,m-j_2-1}\|_{\Cc^{-\eps}}^{\frac{p-2+2m}{2m-1-j_1-j_2}}.
\end{split}
\label{bddY7}
\end{align}
Thus, by \eqref{bddY4}, \eqref{bddY5}, \eqref{bddY6} and \eqref{bddY7}, summing over $0\le j_1\le m$ and $0\le j_2\le m-1$ with $(j_1,j_2)\neq (m,m-1)$, and integrating in $t' \in [t_0,t]$, we deduce
\begin{align}
\begin{split}
&\big| \int_{t_0}^t\langle F_0(v, \Xi)(t'), |v|^{p-2}v(t') \rangle_{L^2(\M)}dt' \big|\\
& \le C_{\circ }\delta_1 \int_{t_0}^t\int_\M|v(t')|^{p-2+2m}dVdt'+C_{\circ}\delta_2\int_{t_0}^t\int_\M|v|^{p-2}|\nabla v|^2(t')dVdt'\\
&\qquad\qquad+C(\dl_1, \dl_2)\big((t-t_0)+\|\Xi\|_{\mathcal Z_T^{\eps, q}}^q\big),
\end{split}
\label{bddY7b}
\end{align}
for some constant $C_{\circ} = C_{\circ}(m)>0$ independent of $\delta_1,\delta_2$ and some large $q = q(m,p) \ge 1$.

We now consider the remaining terms on the right-hand side of \eqref{aprioriestimate}. By the Cauchy-Schwarz, H\"older and Young inequalities, we have (omitting the time dependence again)
\begin{align}
\begin{split}
&\big|\langle |v|^{p-2}vA,dv\rangle_{L^2(T^*\M)}\big|\\
&\le \|A\|_{L^\infty}\bigg(\int_\M|v|^{p-2}|\nabla v|^2dV\bigg)^\frac12\Big(\int_\M |v|^pdV\Big)^\frac12\\
&\le \delta_2\int_\M|v|^{p-2}|\nabla v|^2dV+C(\delta_2,\|A\|_{L^\infty},p,m)\Big(\int_\M |v|^{p-2+2m}dV\Big)^{\frac{p}{p-2+2m}}\\
&\le \delta_2\int_\M|v|^{p-2}|\nabla v|^2dV+\delta_1\int_\M |v|^{p-2+2m}dV+C(\delta_1,\delta_2,\|A\|_{L^\infty},p,m).
\end{split}
\label{bddY8}
\end{align}
Similarly, by integration by parts, the Leibniz rule for 1-forms, Cauchy-Schwarz, H\"older, and Young inequalities, we have
\begin{align}
\begin{split}
&\big|\langle d^*(Av,|v|^{p-2}v\rangle_{L^2(\M)}\big|\les_p \|A\|_{L^\infty}  \int_\M |v|^{p-1}|\nabla v|dV\\
&\les \|A\|_{L^\infty}\Big(\int_\M|v|^{p-2}|\nabla v|^2(t)dV\Big)^\frac12\Big(\int_\M |v|^pdV\Big)^\frac12\\
&\le \delta_2\int_\M|v|^{p-2}|\nabla v|^2dV+\delta_1\int_\M |v|^{p-2+2m}dV+C(\delta_1,\delta_2,\|A\|_{L^\infty},p,m).
\end{split}
\label{bddY9}
\end{align}

By Lemma \ref{LEM:a priori estimate}, \eqref{aprioriestimate}  together with \eqref{bddY7b}, \eqref{bddY8} and \eqref{bddY9}, we end up with
\begin{align*}
&\frac{1}{p} \Big( \big\| v(t) \big\|_{L^p}^p -  \big\| v(t_0) \big\|_{L^p} ^p \Big)  + \beta_1\int_{t_0}^t \int_\M |v(t')|^{p-2 + 2m}dVdt'+\alpha_1\int_{t_0}^t\int_\M|v(t')|^p|A|^2dVdt'\notag\\
&\qquad\qquad\qquad\qquad+ 4 \eta \alpha_1 \int_{t_0}^t \int_\M  |v|^{p -2}(t')| \nb v|^2(t')dV dt' \notag \\
&\leq  (C_{\circ}+2)\delta_1\int_{t_0}^t \int_\M|v(t')|^{p-2+2m}dVdt' +(C_{\circ}+2)\delta_2\int_{t_0}^t\int_\M|v(t')|^{p-2}|\nabla v(t')|^2dVdt'\\
&\qquad\qquad+C(\delta_1,\delta_2)\big((t-t_0)+\|\Xi\|_{\mathcal Z_T^{\eps, q}}^q\big) dt'.
\end{align*}
This proves \eqref{apriori2} after taking $\delta_1,\delta_2$ small enough such that we can absorb the first two terms of the right-hand side above into the left-hand side.
\end{proof}

We are now in position to prove Proposition~\ref{PROP:deterministic}.

\begin{proof}[Proof of Proposition \ref{PROP:deterministic}]
Fix $\al_1, \al_2, \be_1, \be_2, m, s_0, r, p$ as in the statement of Proposition \ref{PROP:deterministic} and let $\eps, q$ be given by Lemma \ref{LEM:a priori 2}. As discussed in the above, we assume without loss of generality that $\al_2 \ge 0$. Note that the condition \eqref{condition on r} is such that Proposition \ref{PROP:LWP2} and condition \eqref{cond a priori} hold; see Remark \ref{RMK:expo}.

Let $v_0 \in \Cc^{s_0}(\M)$ and $\Xi \in \mathcal Z^{\eps, q}(\R_+)$. Consider the solution $v \in C(  [0,T]; \Cc^{s_0}(\M)) \cap  C(  (0,T]; \Cc^{2\e }(\M))$, for some $T \in (0,1]$, to \eqref{model equation} with enhanced data set $(v_0, \Xi)$, provided by Proposition \ref{PROP:LWP1}. 

We now fix an arbitrary time $T_0>0$. By applying Proposition \ref{PROP:LWP2} to \eqref{model equation} with data $(v(T), \Xi(\cdot - T))$, the maximal existence time $T^* \in(T,T_0]$ of $v$ on $[0,T_0]$ satisfies the following alternative:
\begin{align*}
\T=T_0 \qquad\text{or}\qquad\sup_{0\le t<\T}\|v(t)\|_{L^p}=+\infty.
\end{align*}
Thus, if suffices to show that there exists a positive constant $B_0 = B(T_0, \|v_0\|_{\Cc^{s_0}}, \|\Xi\|_{\mathcal Z^{\eps,q_0}_{T_0}})$ such that
\begin{align}
\sup_{0\le t<\T}\|v(t)\|_{L^p}\le B_0
\label{aprioY1}
\end{align}
to conclude the proof since $T_0$ is arbitrary. We note that \eqref{aprioY1} also proves the bound \eqref{boundY}.

The estimate \eqref{aprioY1} would follow immediately from Lemmas \ref{LEM:a priori estimate} and \ref{LEM:a priori 2} if $v$ were smooth. We circumvent this technicality via an approximation argument. To this end, we introduce some notations. Let $N \in \N$ and $\Xi^N = \P_N \Xi$, where $\P_N$ is the frequency projection \eqref{PN}. Then, by Proposition \ref{PROP:LWP1} and the uniform boundedness of $\{\P_N\}_{N\in\N}$ in $L^p$ spaces, there exists $T_1 \le T$ independent of $N$ and a solution $v^N \in C([0,T_1]; \Cc^{s_0}(\M)) \cap  C((0,T_1]; \Cc^{2\e }(\M))$ to \eqref{model equation} with data $(v_0, \Xi^N)$ such that 
\begin{align}\|v^N(T_1)\|_{\Cc^{2\eps}} \les 1+ \|v_0\|_{\Cc^{s_0}}.
\label{bddpf1}
\end{align}
The bound \eqref{bddpf1} is a consequence of the regularity of the solution map \eqref{solmap1}. Here, the implicit constant is independent of $N$.

Now, by Proposition \ref{PROP:LWP2}, there exists a maximal existence time $T_N^* \in (0,T_0]$ of $v^N$ on $[0,T_0]$ such that $v^N \in C([0,T_N^*]; \Cc^{s_0}(\M)) \cap  C((T_1, T^*_N]; \Cc^{2 \eps}(\M))$ and
\begin{align*}
T^*_N=T_0 \qquad\text{or}\qquad\sup_{0\le t<T^*_N}\|v^N(t)\|_{L^p}=+\infty.
\end{align*}
Since, $\Xi^N$ is a smooth function (in space) it is easy to see that $v^N \in  C([0,T_N^*]; \Cc^{s_0}(\M)) \cap  C_t^0 C^2_x((T_1,T_N^*)\times \M)$. Furthemore, $v^N$ satisfies the integral equation
\[v^N(t) = S(t) v_0 + \int_{0}^t S(t-t') F(v^N, \Xi^N) dt'\]
for all $t \in (T_1,T_N^*)$ and where $F$ is as in \eqref{def F}. The last formula shows that $v^N \in C_t^1 C^2_x((T_1,T_N^*)\times \M)$ since $S$ is smoothing in time and $S(t-t') F(v^N, \Xi^N) \in C_{t,t'}^0 C^2_x((T_1,T_N^*)^2\times \M)$.

Hence, Lemmas \ref{LEM:a priori estimate} and \ref{LEM:a priori 2} rule out the possibility of blowup for $v^N$ on $[0,T^*_N]$ and yield
\begin{align}
\begin{split}
\|v^N(t)\|_{L^p} & \leq \| v^N(T_1)  \|_{L^p} + C\big((t-T_1)+\|\Xi^N\|_{\mathcal Z_{T_0}^{\eps, q}}^q\big) \\
& \le C_1 \|v_0\|_{\Cc^{s_0}} + C\big((t-T_1)+\|\Xi\|_{\mathcal Z_{T_0}^{\eps, q}}^q\big) =: B_0
\end{split}
\label{bddpf2}
\end{align}
for all $t \in (T_1,T_N^*)$, where we used \eqref{bddpf1} and Corollary \ref{COR : bernstein}. Here, $C$ is the constant provided by Lemma \ref{LEM:a priori 2} and $C_1 >0$ is an absolute constant which is large enough.

Thus, \eqref{aprioY1} follows from \eqref{bddpf2}, the convergence 
\[ \Xi^N \too \Xi, \quad \text{in $\mathcal Z^{\eps,q}([0,T_0])$} \]
and the regularity of the solution map \eqref{solmap2}.
\end{proof}

We conclude this section by the proof of Theorem \ref{THM : deterministic GWP}.

\begin{proof}[Proof of Theorem \ref{THM : deterministic GWP}]
The proof follows from Propositions \ref{PROP:deterministic} and \ref{PROP:construction stochastic objects} and is an immediate modification of the proof of Theorem \ref{THM : LWP theory 1}. We omit details.
\end{proof}

\end{document}